\documentclass{amsart}
\usepackage{amscd}
\usepackage{amsmath,empheq}
\usepackage{amsfonts}
\usepackage{amssymb}
\usepackage{mathrsfs}

\usepackage[all]{xy}

\newtheorem{theorem}{Theorem}
\newtheorem{ex}[theorem]{Example}
\newtheorem{lemma}[theorem]{Lemma}
\newtheorem{prop}[theorem]{Proposition}
\newtheorem{remark}{Remark}
\newtheorem{corollary}[theorem]{Corollary}
\newtheorem{claim}{Claim}

\newtheorem{definition}[theorem]{Definition}

\newenvironment{proof-sketch}{\noindent{\bf Sketch of Proof}\hspace*{1em}}{\qed\bigskip}

\everymath{\displaystyle}

\newcommand{\RR}{\mathbb R}
\newcommand{\NN}{\mathbb N}

\renewcommand{\leq}{\leqslant}

\renewcommand{\geq}{\geqslant}

\begin{document}

\title[Sensitivity analysis for optimal control problems]{Sensitivity analysis for optimal control problems governed by nonlinear evolution inclusions}

\author[N.S. Papageorgiou]{Nikolaos S. Papageorgiou}
\address{Department of Mathematics, National Technical University, 
				Zografou Campus, 15780 Athens, Greece}
\email{\tt npapg@math.ntua.gr}

\author[V.D. R\u{a}dulescu]{Vicen\c{t}iu D. R\u{a}dulescu}
\address{Institute of Mathematics ``Simion Stoilow" of the Romanian Academy, P.O. Box 1-764,
          014700 Bucharest, Romania \&
 Department of Mathematics, University of Craiova, Street A. I. Cuza 13,
          200585 Craiova, Romania}
\email{\tt vicentiu.radulescu@imar.ro}

\author[D.D. Repov\v{s}]{Du\v{s}an D. Repov\v{s}}
\address{Faculty of Education and Faculty of Mathematics and Physics, University of Ljubljana, 
					Kardeljeva Plo\v{s}\v{c}ad 16, SI-1000 Ljubljana, Slovenia}
\email{\tt dusan.repovs@guest.arnes.si}


\begin{abstract}
We consider a nonlinear optimal control problem governed by a nonlinear evolution inclusion and depending on a parameter $\lambda$. First we examine the dynamics of the problem and establish the nonemptiness of the solution set and produce continuous selections of the solution multifunction $\xi\mapsto S(\xi)$ ($\xi$ being the initial condition). These results are proved in a very general framework and are of independent interest as results about evolution inclusions. Then we use them to study the sensitivity properties of the optimal control problem. We show that we have Hadamard well-posedness (continuity of the value function) and we establish the continuity properties of the optimal multifunction. Finally we present an application on a nonlinear parabolic distributed parameter system.
\end{abstract}

\maketitle


\section{Introduction}

One of the important problems in optimal control theory, is the study of the variations of the set of optimal state-control pairs and of the value of the problem, when we perturb the dynamics, the cost functional and the initial condition of the problem. Such a sensitivity analysis (also known in the literature as ``variational analysis") is important because it gives information about the tolerances which are permitted in the specification of the mathematical models, it suggests ways to solve parametric problems and also can be useful in the computational analysis of the problem. For infinite dimensional  systems (distributed parameter systems), such investigations were conducted by Buttazzo and Dal Maso \cite{8}, Denkowski and Migorski \cite{13}, Ito and Kunisch \cite{24}, Papageorgiou \cite{30} (linear systems), Papageorgiou \cite{31}, Sokolowski \cite{37} (semilinear systems) and Hu and Papageorgiou \cite{22}, Papageorgiou \cite{33,32} (nonlinear systems). We also mention
  the books of Buttazzo \cite{7}, Dontchev and Zolezzi \cite{17}, Ito and Kunisch \cite{25}, Sokolowski and Zolezio \cite{38} (the latter for shape optimization problems). In this paper we conduct such an analysis for a very general class of systems driven by  nonmonotone evolution inclusions.

So, let $T=[0,b]$ be the time interval and $(X,H,X^*)$ an evolution triple of spaces (see Section 2). We assume that $X\hookrightarrow H$ compactly. The space of controls is modelled by a separable reflexive Banach space $Y$ and $E$ is a compact metric space and corresponds to the parameter space. As we have already mentioned, we consider systems monitored by evolution inclusions. These inclusions represent a way to model systems with deterministic uncertainties, see the books of Aubin and Frankowska \cite{2}, Fattorini \cite{18}, and Roubicek \cite{36}.

The problem under consideration is the following:
\begin{eqnarray}\label{eq1}
	\left\{\begin{array}{l}
		J(x,u,\lambda)=\int^b_0L(t,x(t),\lambda)dt+\int^b_0H(t,u(t),\lambda)dt+\hat{\psi}(\xi,x(b),\lambda)\rightarrow \inf=\\ \hspace{10cm}m(\xi,\lambda),\\
		-x'(t)\in A_{\lambda}(t,x(t))+F(t,x(t),\lambda)+G(t,u(t),\lambda)\ \mbox{for almost all}\ t\in T,\\
		x(0)=\xi,u(t)\in U(t,\lambda)\ \mbox{for almost all}\ t\in T,\lambda\in E.
	\end{array}\right\}
\end{eqnarray}

In this problem
$$A_{\lambda}:T\times X\rightarrow 2^{X^*}\ \mbox{for every}\ \lambda>0,\ F:T\times H\times E\rightarrow 2^H\setminus\{\emptyset\},\ G:T\times Y\times E\rightarrow 2^H\setminus\{\emptyset\}$$
and the precise conditions on them will be given in Section 4. For every initial state $\xi\in H$ and every parameter $\lambda\in E$,  we denote the set of admissible state-control pairs (that is, pairs $(x,u)$ which satisfy the dynamics and the constraints of problem (\ref{eq1})) by $Q(\xi,\lambda)$. We investigate the dependence  of $Q(\xi,\lambda)$ on the two variables $(\xi,\lambda)\in H\times E$. Also, $\Sigma(\xi,\lambda)$ denotes the set of optimal state-control pairs (that is, $(x^*,u^*)\in\Sigma(\xi,\lambda)$ such that $J(x^*,u^*,\xi,\lambda)=m(\xi,\lambda)$). So, $\Sigma(\xi,\lambda)\subseteq Q(\xi,\lambda)$. We establish the nonemptiness of the set $\Sigma(\xi,\lambda)$ and examine the continuity properties of the value function $(\xi,\lambda)\mapsto m(\xi,\lambda)$ and of the multifunction $(\xi,\lambda)\mapsto\Sigma(\xi,\lambda)$.

The nonemptiness and other continuity and structural properties of the set $Q(\xi,\lambda)$ are consequences of general results about evolution inclusions which we prove in Section 3 and which are of independent interest. The class of evolution inclusions considered in Section 3 is more general than the classes studied by Chen, Wang and Zu \cite{11}, Denkowski, Migorski and Papageorgiou \cite{16}, Liu \cite{28}, Papageorgiou and Kyritsi \cite{34}.

In the next section, for the convenience of the reader, we review the main mathematical tools which we will need in this paper.

\section{Mathematical Background}

Suppose that $V$ and $Z$ are Banach spaces and assume that $V$ is embedded continuously and densely into $Z$ (denoted by $V\hookrightarrow Z$). Then it is easy to check that
\begin{itemize}
	\item	$Z^*$ is embedded continuously into $V^*$;
	\item if $V$ is reflexive, then $Z^*\hookrightarrow V^*$.
\end{itemize}

Having this observation in mind, we can introduce the notion of evolution triple of spaces, which is central in the class of evolution equations considered here.
\begin{definition}\label{def1}
	A triple $(X,H,X^*)$ of spaces is said to be an ``evolution triple" (or ``Gelfand triple" or ``spaces in normal position"), if the following hold:
	\begin{itemize}
		\item[(a)] $X$ is a separable reflexive Banach space and $X^*$ is its topological dual;
		\item[(b)] $H$ is a separable Hilbert space identified with its dual $H^*=H$ (pivot space);
		\item[(c)] $X\hookrightarrow H$.
	\end{itemize}
\end{definition}

According to the remark made in the beginning of this section, we also have $H^*=H\hookrightarrow X^*$. In this paper we also assume that the embedding of $X$ into $H$ is  compact. Hence by Schauder's theorem  (see, for example, Gasinski and Papageorgiou \cite[Theorem 3.1.22, p. 275]{20}), so is the embedding of $H^*=H$ into $X^*$. In what follows, by $||\cdot||$ (resp. $|\cdot|,\ ||\cdot||_*$) we denote the norm of the space $X$ (resp. $H,X^*$). By $\left\langle \cdot,\cdot\right\rangle$ we denote the duality brackets for the pair $(X^*,X)$ and by $(\cdot,\cdot)$ the inner product of the Hilbert space $H$. We know that
$$\left.\left\langle \cdot,\cdot\right\rangle\right|_{H\times X}=(\cdot,\cdot)\,.$$

Also, let $\beta>0$ be such that
\begin{equation}\label{eq2}
	|\cdot|\leq\beta||\cdot||\,.
\end{equation}

We introduce the following space which has a central role in the study of the evolution inclusions. So, let $1<p<\infty$ and set
$$W_p(0,b)=\{x\in L^p(T,X):x'\in L^{p'}(T,X^*)\}\ \left(\frac{1}{p}+\frac{1}{p'}=1\right)\,.$$

In this definition the derivative of $x$ is understood in the sense of vectorial distributions (weak derivative). In fact, if we view $x$ as an $X^*$-valued function, then $x(\cdot)$ is absolutely continuous, hence strongly differentiable almost everywhere. Therefore
$$W_p(0,b)\subseteq AC^{1,p'}(T,X^*)=W^{1,p'}((0,b),X^*)\,.$$

The space $W_p(0,b)$, equipped with the norm
$$||x||_{W_p}=||x||_{L^p(T,X)}+||x'||_{L^{p'}(T,X^*)}\ \mbox{for all}\ x\in W_p(0,b),$$
becomes a separable reflexive Banach space. We know that
\begin{eqnarray}
	&&W_p(0,b)\hookrightarrow C(T,H),\label{eq3}\\
	&&W_p(0,b)\hookrightarrow L^p(T,H)\ \mbox{compactly}.\label{eq4}
\end{eqnarray}

The following integration by parts formula is very helpful:
\begin{prop}\label{prop2}
	If $x,y\in W_p(0,b)$, then $t\mapsto (x(t),y(t))$ is absolutely continuous and
	$$\frac{d}{dt}(x(t),y(t))=\left\langle x'(t),y(t)\right\rangle+\left\langle x(t),y'(t)\right\rangle\ \mbox{for almost all}\ t\in T.$$
\end{prop}

We know that for all $1\leq p< \infty$,
$$L^p(T,X)^*=L^{p'}(T,X^*)$$
with $p'=+\infty$ if $p=1$ (see Gasinski and Papageorgiou \cite[Theorem 2.2.9, p. 129]{20}).

Now, let $(\Omega,\Sigma)$ be a measurable space and $V$ a separable Banach space. We introduce the following hyperspaces:
\begin{eqnarray*}
	&&P_{f(c)}(V)=\{C\subseteq V:C\ \mbox{is nonempty, closed, (convex)}\}\\
	&&P_{(w)k(c)}(V)=\{C\subseteq V:C\ \mbox{is nonempty, (weakly-)compact, (convex)}\}\,.
\end{eqnarray*}

Given a multifunction $F:\Omega\rightarrow 2^V\backslash\{\emptyset\}$, the ``graph" of $F$ is the set
$${\rm Gr}\,F=\{(\omega,v)\in\Omega\times V:v\in F(\omega)\}\,.$$

We say that $F(\cdot)$ is ``graph measurable" if ${\rm Gr}\,F\in\Sigma\times B(V)$ with $B(V)$ being the Borel $\sigma$-field of $V$. If $\mu(\cdot)$ is a $\sigma$-finite measure on $\Sigma$ and $F:\Omega\rightarrow 2^V\backslash\{\emptyset\}$ is graph measurable, then the Yankov-von Neumann-Aumann selection theorem (see Hu and Papageorgiou \cite[Theorem 2.14, p. 158]{23}) implies that $F(\cdot)$ admits a measurable selection, that is, there exists a $\Sigma$-measurable function $f:\Omega\rightarrow V$ such that $f(\omega)\in F(u)$ $\mu$-almost everywhere. In fact, there is a whole sequence $\{f_n\}_{n\geq 1}$ of such measurable selections such that $F(\omega)\subseteq\overline{\{f_n(\omega)\}}$ $\mu$-almost everywhere (see Hu and Papageorgiou \cite[Proposition 2.17, p. 159]{23}). Moreover, the above results are valid if $V$ is only a Souslin space. Recall that a Souslin space need not be metrizable (see Gasinski and Papageorgiou \cite[p. 232]{21}). A multifunction $F:\Omega\
 rightarrow P_f(V)$ is said to be ``measurable" if for all $y\in V$, the function $$\omega\mapsto d(y,F(\omega))=\inf[||y-v||_V:v\in F(\omega)]$$ is $\Sigma$-measurable. A multifunction $F:\Omega\rightarrow P_f(V)$ which is measurable is also graph measurable. The converse is true if ($\Omega,\Sigma$) admits a complete $\sigma$-finite measure $\mu$. If $(\Omega,\Sigma,\mu)$ is a $\sigma$-finite measure space and $F:\Omega\rightarrow 2^V\backslash\{\emptyset\}$ is a multifunction, then for $1\leq p\leq\infty$ we introduce the set
$$S^P_F=\{f\in L^p(\Omega,Y):f(\omega)\in F(\omega)\ \mu\mbox{-a.e.}\}.$$

Evidently, $S^P_F\neq\emptyset$ if and only if $\omega\mapsto \inf[||v||_V:v\in F(\omega)]$ belongs to $L^p(\Omega)$. Moreover, the set $S^P_F$ is ``decomposable", that is, if $(A,f_1,f_2)\in\Sigma\times S^P_F\times S^P_F$, then
$$\chi_{A}f_{1}+\chi_{\Omega\backslash A}f_2\in S^P_F.$$

Here, for $C\in\Sigma$, by $\chi_C$ we denote the characteristic function of the set $C\in\Sigma$.

For every $D\subseteq \Sigma,\ D\neq\emptyset$, we define
\begin{eqnarray*}
	&&|D|=\sup[||v||_V:v\in D]\\
	\mbox{and}&&\sigma(v^*,D)=\sup[\left\langle v^*,v\right\rangle_V:v\in D]\ \mbox{for all}\ v^*\in V^*.
\end{eqnarray*}
Here, $\left\langle \cdot,\cdot\right\rangle_V$ denotes the duality brackets of the pair $(V^*,V)$. The function $\sigma(\cdot,D):V^*\rightarrow \bar{\RR}=\RR\cup\{+\infty\}$ is known as the ``support function" of $D$.

Let $Z,W$ be Hausdorff topological spaces. We say that a multifunction $G:Z\rightarrow 2^W\backslash\{\emptyset\}$ is ``upper semicontinuous" (usc for short), respectively ``lower semicontinuous" (lsc for short), if for all $U\subseteq W$ open, the set
$$G^+(U)=\{z\in Z:G(z)\subseteq U\}\ \mbox{respectively}\ G^-(U)=\{z\in Z:G(z)\cap U\neq\emptyset\}$$
is open in $Z$. If $G(\cdot)$ is both usc and lsc, then we say that $G(\cdot)$ is continuous. On a Hausdorff topological space $(W,\tau)$ ($\tau$ being the Hausdorff topology), we can define a new topology $\tau_{seq}$ whose closed sets are the sequentially $\tau$-closed sets. Then topological properties with respect to this topology have the prefix ``sequential". Note that $\tau\subseteq\tau_{seq}$ and the two are equal, if $\tau$ is first countable (see Buttazzo \cite[p. 9]{7} and Gasinski and Papageorgiou \cite[p. 808]{21}). We say that $G:Z\rightarrow 2^W\backslash\{\emptyset\}$ is ``closed", if the graph ${\rm Gr}\,G\subseteq Z\times W$ is closed.

For any Banach space $V$, on $P_f(V)$ we can define a generalized metric, known as the ``Hausdorff metric", by setting
$$h(E,M)=\max\left[\sup\limits_{e\in E}d(e,M),\sup\limits_{m\in M}d(m,E)\right]\,.$$

Recall that $(P_f(V),h)$ is a complete metric space (see Hu and Papageorgiou \cite[p. 6]{23}). If $Z$ is a Hausdorff topological space, a multifunction $G:Z\rightarrow P_f(V)$ is said to be ``$h$-continuous", if it is continuous from $Z$ into $(P_f(V),h)$.

Also, if $E,M\subseteq V$ are nonempty, bounded, closed and convex subsets, then
$$h(E,M)=\sup[|\sigma(v^*,E)-\sigma(v^*,M)|:v^*\in V^*,||v^*||_{V^*}\leq 1]$$
(H\"{o}rmander's formula).

Let $(W,\tau)$ be a Hausdorff topological space with topology $\tau$ and let $\{E_n\}_{n\geq 1}\subseteq 2^W\backslash\{\emptyset\}$. We define
 \begin{eqnarray*}
	&&K_{seq}(\tau)-\liminf\limits_{n\rightarrow\infty}E_n=\{y\in W:y=\tau-\lim\limits_{n\rightarrow\infty}y_n,y_n\in E_n\ \mbox{for all}\ n\in\NN\},\\
	&&K_{seq}(\tau)-\limsup\limits_{n\rightarrow\infty}E_n=\{y\in W:y=\tau-\lim\limits_{n\rightarrow\infty}y_{n_k},y_{n_k}\in E_{n_k},n_1<n_2<\ldots<n_k<\ldots\}.
\end{eqnarray*}

Sometimes we drop the $K_{seq}$-symbol and simply write $\tau-\limsup_{n\rightarrow\infty}E_n$ and $\tau-\liminf\limits_{n\rightarrow\infty}E_n$.

Returning to the setting of an evolution triple, we consider a sequence of multivalued maps $a_n,a:L^p(T,X)\rightarrow 2^{L^{p'}(T,X^*)}\backslash\{\emptyset\}$ ($n\in\NN$) such that for every $h^*\in L^{p'}(T,X^*)$ the inclusions
$$y'+a_n(y)\ni h^*\ (n\in\NN)\ \mbox{and}\ y'+a(y)\ni h$$
have unique solutions $e_n(h^*),e(h^*)\in W_p(0,b)$.

We say that $\frac{d}{dt}+a_n$ ``PG-converges" to $\frac{d}{dt}+a$ (denoted by $\frac{d}{dt}+a_n\xrightarrow{PG}\frac{d}{dt}+a$ as $n\rightarrow\infty$), if for every $h^*\in L^{p'}(T,X^*)$ we have
$$e_n(h^*)\stackrel{w}{\rightarrow}e(h^*)\ \mbox{in}\ W_p(0,b).$$

In what follows, by $X_w$ (respectively $H_w,X^*_w$) we denote the space $X$ (respectively $H,X^*$) furnished with the weak topology. Also, by $|\cdot|_1$ we denote the Lebesgue measure on $\RR$ and by $((\cdot,\cdot))$ we denote the duality brackets for the pair $(L^{p'}(T,X^*),L^p(T,X))$. So, we have
$$((h^*,f))=\int^b_0\left\langle h^*(t),f(t)\right\rangle dt\ \mbox{for all}\ h^*\in L^{p'}(T,X^*),\ \mbox{all}\ f\in L^p(T,X).$$

Next, let us recall some useful facts from the theory of nonlinear operators of monotone type.

So, let $V$ be a reflexive Banach space, $L:D(L)\subseteq V\rightarrow V^*$ a linear maximal monotone operator and $a:V\rightarrow 2^{V^*}$. We say that $a(\cdot)$ is ``$L$-pseudomonotone" if the following conditions hold:
\begin{itemize}
	\item[(a)] For every $v\in V,a(v)\in P_{wkc}(V^*)$.
	\item[(b)] $a(\cdot)$ is bounded (that is, maps bounded sets to bounded sets).
	\item[(c)] If $\{v_n\}_{n\geq 1}\subseteq D(L),\ v_n\stackrel{w}{\rightarrow}v\in D(L)$ in $V$, $L(v_n)\stackrel{w}{\rightarrow}L(v)$ in $V^*$, $v^*_n\in a(v_n)$ for all $n\in\NN$, $v^*_n\stackrel{w}{\rightarrow}v^*$ in $X^*$ and $\limsup\limits_{n\rightarrow\infty}\left\langle v^*_n,v_n-v\right\rangle_V\leq 0,$ then $v^*\in a(v)$ and $\left\langle v^*_n,v_n\right\rangle_V\rightarrow\left\langle v^*,v\right\rangle_V$
\end{itemize}

Such maps have nice surjectivity properties. The next result is due to Papageorgiou, Papalini and Renzacci \cite{35} and it extends an earlier single-valued result of Lions \cite[Theorem 1.2, p. 319]{27}.
\begin{prop}\label{prop3}
	Assume that $V$ is a reflexive Banach space which is strictly convex, $L:D(L)\subseteq V\rightarrow V^*$ is a linear maximal monotone operator and $A:V\rightarrow 2^{V^*}$ is $L$-pseudomonotone and strongly coercive, that is,
	$$\frac{\inf[\left\langle v^*,v\right\rangle_V:v^*\in A(v)]}{||v||_V}\rightarrow+\infty\ \mbox{as}\ ||v||_V\rightarrow+\infty.$$
	Then $R(L+V)=V^*$ (that is, $L+V$ is surjective).
\end{prop}

In the next section we obtain some results about a general class of evolution inclusions, which will help us study  problem (\ref{eq1}) (see Section 4).

\section{Nonlinear Evolution Inclusions}

Let $T=[0,b]$ and let ($X,H,X^*$) be an evolution triple with $X\hookrightarrow H$ compactly (see Definition \ref{def1}). In this section we deal with the following evolution inclusion:
\begin{equation}\label{eq5}
	\left\{\begin{array}{l}
		-x'(t)\in A(t,x(t))+E(t,x(t))\ \mbox{for almost all}\ t\in T,\\
		x(0)=\xi .
	\end{array}\right\}
\end{equation}

The hypotheses on the data of (\ref{eq5}) are the following:

\smallskip
$H(A)_1:\quad A:T\times X\rightarrow 2^{X^*}$ is a map such that
\begin{itemize}
	\item[(i)] for all $x\in X,\ t\mapsto A(t,x)$ is graph measurable;
	\item[(ii)] for almost all $t\in T,\ {\rm Gr}\, A(t,\cdot)$ is sequentially closed in $X_w\times X^*_w$ and $x\mapsto A(t,x)$ is pseudomonotone;
	\item[(iii)] for almost all $t\in T$, all $x\in X$ and all $h^*\in A(t,x)$, we have
	$$||h||_*\leq a_1(t)+c_1||x||^{p-1}$$
	with $2\leq p,\ a_1\in L^{p'}(T)$ and $c_1>0$;
	\item[(iv)] for almost all $t\in T$, all $x\in X$ and all $h^*\in A(t,x)$, we have
	$$\left\langle h^*,x\right\rangle\geq c_2||x||^p-a_2(t),$$
	with $c_2>0,\ a_2\in L^1(T)$.
\end{itemize}	

\begin{remark}
	If $A(\cdot,\cdot)$ is single-valued, then in hypothesis $H(A)(ii)$ we can drop the condition on the graph of ${\rm Gr}\, A(t,\cdot)$ and only assume that for almost all $t\in T,\ x\mapsto A(t,x)$ is pseudomonotone. Similarly, if for almost all $t\in T,A(t,\cdot)$ is maximal monotone. An example of where the condition on the graph of $A(t,\cdot)$ is satisfied is the following. For simplicity we drop the $t$-dependence
	$$A(x)=-{\rm div}\,\partial\varphi(Dx)-{\rm div}\,\xi(Dx)$$
	where $\varphi:L^p(\Omega,\RR^N)\rightarrow\RR$ is continuous, convex and $\xi:L^p(\Omega,\RR^N)\rightarrow\RR$ is continuous and $|\xi(y)|\leq \hat{c}(1+|y|^{\tau-1})$ for all $y\in\RR^N$, some $\hat{c}>0$ and with $1\leq\tau<p$. Then recalling that $W^{1,p}(\Omega)\hookrightarrow W^{1,\tau}(\Omega)$ compactly (see Zeidler \cite[p. 1026]{40}), we easily see that ${\rm Gr}\,A$ is sequentially closed in $W^{1,p}(\Omega)_w\times W^{1,p}(\Omega)^*_w$.
\end{remark}

$H(F)_1:$ $F:T\times H\rightarrow P_{{f_c}}(H)$ is a multifunction such that
\begin{itemize}
	\item[(i)] for all $x\in H$, $t\mapsto F(t,x)$ is graph measurable;
	\item[(ii)] for almost all $t\in T,\ {\rm Gr}\,F(t,\cdot)$ is sequentially closed in $H\times H_w$;
	\item[(iii)] for almost all $t\in T$, all $x\in H$ and all $h\in F(t,x)$
	$$|h|\leq a_3(t)+c_3|x|,$$
	with $a_3\in L^2(T),c_3>0$ and if $p=2$, then $\beta^2c_3<c_2$ (see (\ref{eq2})).
\end{itemize}

By a solution of problem (\ref{eq5}) we understand a function $x\in W_p(0,b)$ such that
$$-x'(t)=h^*(t)+f(t)\ \mbox{for almost all}\ t\in T,$$
with $h^*\in L^{p'}(T,X^*), f\in L^2(T,H)$ such that
$$h^*(t)\in A(t,x(t))\ \mbox{and}\ f(t)\in F(t,x(t))\ \mbox{for almost all}\ t\in T.$$

By $S(\xi)$ we denote the set of solutions of problem (\ref{eq5}). In the sequel we investigate the structure of $S(\xi)$.

Consider the multivalued map $a:L^p(T,X)\rightarrow 2^{L^{p'}(T,X^*)}$ defined by
\begin{equation}\label{eq6}
	a(x)=\{h^*\in L^{p'}(T,X^*):h^*(t)\in A(t,x(t))\ \mbox{for almost all}\ t\in T\}\ \mbox{for all}\ x\in L^p(T,X).
\end{equation}

Note that $a(\cdot)$ has values in $P_{wkc}(L^{p'}(T,X^*))$ (see hypotheses $H(A)(i),(iii)$ and use the Yankov-von Neumann-Aumann selection theorem (see Hu and Papageorgiou \cite[Theorem 2.15, p. 518]{23})).

\begin{lemma}\label{lem4}
	If hypotheses $H(A)_1$ hold, $x_n\stackrel{w}{\rightarrow}x$ in $W_p(0,b)$, $x_n(t)\stackrel{w}{\rightarrow}x(t)$ in $X$ for almost all $t\in T$, $h^*_n\stackrel{w}{\rightarrow}h^*$ in $L^{p'}(T,X^*)$ and $h^*_n\in a(x_n)$ for all $n\in\NN$, then $h^*\in a(x)$.
\end{lemma}
\begin{proof}
	Let $v\in X$ and consider the function $x\mapsto\sigma(v,A(t,x))$ (see Section 2). We will show that it is sequentially upper semicontinuous. To this end we need to show that given $\lambda\in\RR$, the superlevel set
	$$E_{\lambda}=\{x\in X:\lambda\leq\sigma(v,A(t,x))\}$$
	is sequentially closed in $X_w$. So, we consider a sequence $\{x_n\}_{n\geq 1}\subseteq E_{\lambda}$ and assume that
	$$x_n\stackrel{w}{\rightarrow}x\ \mbox{in}\ X.$$
	
	Let $h^*_n\in A(t,x_n)$ ($n\in\NN$) be such that
	\begin{equation}\label{eq7}
		\left\langle h^*_n,v\right\rangle=\sigma(v,A(t,x_n))\ \mbox{for all}\ n\in\NN
	\end{equation}
	(recall $A(t,x_n)\in P_{wkc}(X^*)$). Evidently, $\{h^*_n\}_{n\geq 1}\subseteq X^*$ is bounded (see hypothesis\\ $H(A)_1(iii)$) and so by passing to a subsequence if necessary, we may assume that
	\begin{eqnarray}\label{eq8}
		&&h^*_n\stackrel{w}{\rightarrow}h^*\ \mbox{in}\ X^*,\nonumber\\
		&\Rightarrow&h^*\in A(t,x)\ (\mbox{see hypothesis}\ H(A)_1(ii)).
	\end{eqnarray}
	
	Then we have
	\begin{eqnarray*}
		&&\lambda\leq\left\langle h^*,v\right\rangle\leq\sigma(v,A(t,x))\ (\mbox{see (\ref{eq7}),(\ref{eq8})}),\\
		&\Rightarrow&x\in E_{\lambda}.
	\end{eqnarray*}
	
	This proves the upper semicontinuity of the map $x\mapsto\sigma(v,A(t,x))$.
	
	Now let $v\in L^p(T,X)$. We have
	\begin{eqnarray*}
		&&((h^*_n,v))\leq\sigma(v,a(x_n))=\int^b_0\sigma(v(t),A(t,x_n(t)))dt\ \mbox{for all}\ n\in\NN\\
		&&\hspace{1cm}(\mbox{see Hu and Papageorgiou \cite[Theorem 3.24, p. 183]{23}}),\\
		 &\Rightarrow&((h^*,v))\leq\limsup\limits_{n\rightarrow\infty}\sigma(v,a(x_n))\leq\int^b_0\limsup\limits_{n\rightarrow\infty}\sigma(v(t),A(t,x_n(t)))dt\\
		&&\hspace{1cm}(\mbox{by Fatou's lemma})\\
		&&\leq\int^b_0\sigma(v(t),A(t,x(t)))dt\\
		&&\hspace{1cm}(\mbox{by the first part of the proof and since by hypothesis}\ x_n(t)\stackrel{w}{\rightarrow}x(t)\ \mbox{in}\ X)\\
		&&=\sigma(v,a(x)),\\
		&&\Rightarrow h^*\in a(x).
	\end{eqnarray*}
\end{proof}

\begin{lemma}\label{lem5}
	If hypotheses $H(A)_1$ hold, then the multivalued map $a:L^p(T,X)\rightarrow 2^{L^{p'}(T,X^*)}$ defined by (\ref{eq6}) is $L$-pseudomonotone.
\end{lemma}
\begin{proof}
	Suppose $x_n\stackrel{w}{\rightarrow}x$ in $W_p(0,b),h^*_n\in a(x_n)$ for all $n\in\NN$, $h^*_n\stackrel{w}{\rightarrow}h^*$ in $L^{p'}(T,X^*)$ and
	\begin{equation}\label{eq9}
		\limsup\limits_{n\rightarrow\infty}((h^*_n,x_n-x))\leq 0.
	\end{equation}
	
	From (\ref{eq3}) we infer that
	\begin{equation}\label{eq10}
		x_n(t)\stackrel{w}{\rightarrow}x(t)\ \mbox{in}\ H\ \mbox{for all}\ t\in T\ \mbox{as}\ n\rightarrow\infty\,.
	\end{equation}
	
	We set $\vartheta_n(t)=\left\langle h^*_n(t),x_n(t)-x(t)\right\rangle$. Let $N$ be the Lebesgue-null set in $T=[0,b]$ outside of which hypotheses $H(A)_1(ii),(iii),(iv)$ hold. Using hypotheses $H(A)_1(iii),(iv)$ we have
	\begin{equation}\label{eq11}
		\vartheta_n(t)\geq c_2||x_n(t)||^2-a_2(t)-(a_1(t)+c_1||x_n(t)||^{p-1})||x(t)||\ \mbox{for all}\ t\in T\backslash N.
	\end{equation}
	
	We introduce the Lebesgue measurable set $D\subseteq T$ defined by
	$$D=\{t\in T:\liminf\limits_{n\rightarrow\infty}\vartheta_n(t)<0\}.$$
	
	Suppose that $|D|_1>0$. If $t\in D\cap(T\backslash N)$, then from (\ref{eq11}) we see that
	$$\{x_n(t)\}_{n\geq 1}\subseteq X\ \mbox{is bounded}.$$
	
	Then it follows from (\ref{eq10}) that
	\begin{equation}\label{eq12}
		x_n(t)\stackrel{w}{\rightarrow}x(t)\ \mbox{in $X$ for all}\ t\in D\cap(T\backslash N).
	\end{equation}
	
	We fix  $t\in D\cap(T\backslash N)$ and choose a subsequence $\{n_k\}$ of $\{n\}$ (in general this subsequence depends on $t$) such that
	$$\lim\limits_{k\rightarrow\infty}\vartheta_{n_k}(t)=\liminf\limits_{n\rightarrow\infty}\vartheta_n(t).$$
	
	By hypothesis $H(A)_1(ii)$, $A(t,\cdot)$ is pseudomonotone and since $t\in D$, we infer that
	$$\lim\limits_{k\rightarrow\infty}\left\langle h^*_{n_k}(t),x_{n_k}(t)-x(t)\right\rangle=0,$$
	a contradiction. So, $|D|_1=0$ and we have
	\begin{equation}\label{eq13}
		0\leq\liminf\limits_{n\rightarrow\infty}\vartheta_n(t)\ \mbox{for almost all}\ t\in T.
	\end{equation}
	
	Invoking the extended Fatou's lemma (see Denkowski, Migorski and Papageorgiou \cite[Theorem 2.2.33, p. 145]{15}), we have
	\begin{eqnarray}\label{eq14}
		0&\leq&\int^b_0\liminf\limits_{n\rightarrow\infty}\vartheta_n(t)dt\ (\mbox{see (\ref{eq13})})\nonumber\\
		&\leq&\liminf\limits_{n\rightarrow\infty}\int^b_0\vartheta_n(t)dt\nonumber\\
		&\leq&\limsup\limits_{n\rightarrow\infty}\int^b_0\vartheta_n(t)dt\nonumber\\
		&=&\limsup\limits_{n\rightarrow\infty}\int^b_0\left\langle h^*_n(t),x_n(t)-x(t)\right\rangle dt\nonumber\\
		&=&\limsup\limits_{n\rightarrow\infty}((h^*_n,x_n-x))\leq 0\ (\mbox{see (\ref{eq9})}),\nonumber\\
		\Rightarrow&&\int^b_0\vartheta_n(t)dt\rightarrow 0.
	\end{eqnarray}
	
	We write
	\begin{equation}\label{eq15}
		|\vartheta_n(t)|=\vartheta^+_n(t)+\vartheta^-_n(t)=\vartheta_n(t)+2\vartheta^-_n(t).
	\end{equation}
	
	Note that
	\begin{equation}\label{eq16}
		\vartheta^-_n(t)\rightarrow 0\ \mbox{for almost all}\ t\in T\ (\mbox{see (\ref{eq13})}).
	\end{equation}
	
	Moreover, from (\ref{eq11}) we have
	$$\vartheta_n(t)\geq\eta_n(t)\ \mbox{for almost all}\ t\in T,\ \mbox{all}\ n\in\NN,$$
	with $\{\eta_n\}_{n\geq 1}\subseteq L^1(T)$ uniformly integrable. Then
	$$\vartheta_n^-(t)\leq\eta^-_n(t)\ \mbox{for almost all}\ t\in T,\ \mbox{all}\ n\in\NN,$$
	with $\{\eta^-_n\}_{n\geq 1}\subseteq L^1(T)$ uniformly integrable. Using (\ref{eq16}) and invoking Vitali's theorem we infer that
	\begin{eqnarray}\label{eq17}
		&&\vartheta^-_n\rightarrow 0\ \mbox{in}\ L^1(T),\nonumber\\
		&\Rightarrow&\vartheta_n\rightarrow 0\ \mbox{in}\ L^1(T)\ (\mbox{see (\ref{eq14}), (\ref{eq15})}).
	\end{eqnarray}
	
	Then we have
	\begin{eqnarray*}
		&&|((h^*_n,x_n))-((h^*,x))|\\
		&&\leq|((h^*_n,x_n-x))|+|((h^*_n-h^*,x))|\rightarrow 0\\
		&&(\mbox{see (\ref{eq17}) and recall that}\ h^*_n\stackrel{w}{\rightarrow}h^*\ \mbox{in}\ L^{p'}(T,X^*)),\\
		&\Rightarrow&((h^*_n,x_n))\rightarrow((h^*,x)).
	\end{eqnarray*}
	
	In addition, from (\ref{eq12}) and Lemma \ref{lem4}, we have that
	$$h^*\in a(x).$$
	
	This proves the $L$-pseudomonotonicity of $a(\cdot)$.
\end{proof}

\begin{remark}
	From the above proof it is clear why in the case of a single-valued map $A(t,x)$, in hypothesis $H(A)_1(ii)$ we can drop the condition on the graph of $A(t,\cdot)$ and only assume that for almost all $t\in T$ $x\mapsto A(t,x)$ is pseudomonotone. Indeed, in this case, from (\ref{eq17}) we have (at least for a subsequence) that
	\begin{eqnarray*}
		&&\vartheta_n(t)\rightarrow 0\ \mbox{for almost all}\ t\in T,\\
		&\Rightarrow&A(t,x_n(t))\stackrel{w}{\rightarrow}A(t,x(t))\ \mbox{for almost all}\ t\in T\ \mbox{in}\ X^*\ (\mbox{since}\ A(t,\cdot)\ \mbox{is pseudomonotone}).
	\end{eqnarray*}

	In the multivalued case, there is no canonical way to identify the pointwise limit of the sequence $\{h^*_n(t)\}_{n\geq 1}\subseteq X^*$. If for almost all $t\in T,A(t,\cdot)$ is maximal monotone, then again, we do not need the graph hypothesis on $A(t,\cdot)$. In this case $a(\cdot)$ is also maximal monotone and then the lemma is a consequence of (\ref{eq9}) and Lemma 1.3, p. 42 of Barbu \cite{4}. It is worth mentioning that a similar strengthening of the topology in the range space was used by Defranceschi \cite{12}, while studying $G$-convergence of multivalued operators.
\end{remark}

Without loss of generality, invoking the Troyanski renorming theorem (see Gasinski and Papageorgiou \cite[Remark 2.115, p. 241]{21}), we may assume that both $X$ and $X^*$ are locally uniformly convex, hence $L^p(T,X)$ and $L^{p'}(T,X^*)$ are strictly convex.

We are now ready for the first result concerning the solution set $S(\xi)$.
\begin{theorem}\label{th6}
	If hypotheses $H(A)_1,H(F)_1$ hold and $\xi\in H$, then the solution set $S(\xi)$ is nonempty, weakly compact in $W_p(0,b)$ and compact in $C(T,H)$.
\end{theorem}
\begin{proof}
	First suppose that $\xi\in X$. We define
	$$A_1(t,x)=A(t,x+\xi)\ \mbox{and}\ F_1(t,x)=F(t,x+\xi).$$
	
	Evidently, $A_1(t,x)$ and $F_1(t,x)$ have the same measurability, continuity and growth properties as the multivalued maps $A(t,x)$ and $F(t,x)$. So, we may equivalently consider the following Cauchy problem
	\begin{equation}\label{eq18}
		\left\{\begin{array}{l}
			-x'(t)\in A_1(t,x(t))+F_1(t,x(t))\ \mbox{for almost all}\ t\in T,\\
			x(0)=0.
		\end{array}\right\}
	\end{equation}
	
	Note that if $x\in W_p(0,b)$ is a solution of (\ref{eq18}), then $\hat{x}=x-\xi$ is a solution of (\ref{eq5}) (when $\xi\in X$, that is, the initial condition is regular). Consider the linear densely defined operator $L:D(L)\subseteq L^p(T,X)\rightarrow L^{p'}(T,X^*)$ defined by
	$$L(x)=x'\ \mbox{for all}\ x\in W^0_p(0,b)=\{y\in W_p(0,b):y(0)=0\}$$
	(the evaluation $y(0)=0$, makes sense by virtue of (\ref{eq3})).
	
	Consider the multivalued maps $a_1,G_1:L^p(T,X)\rightarrow 2^{L^{p'}(T,X^*)}\backslash\{\emptyset\}$ defined by
	\begin{eqnarray*}
		&&a_1(x)=\{h^*\in L^{p'}(T,X^*):h^*(t)\in A_1(t,x(t))\ \mbox{for almost all}\ t\in T\},\\
		&&G_1(x)=\{f\in L^{p'}(T,X^*):f(t)\in F_1(t,x(t))\ \mbox{for almost all}\ t\in T\}.
	\end{eqnarray*}
	
	We set $K(x)=a_1(x)+G_1(x)$ for all $x\in L^p(T,X)$. Then
	$$K:L^p(T,X)\rightarrow 2^{L^{p'}(T,X^*)}\backslash\{\emptyset\}.$$
	\begin{claim}\label{cl1}
		$K$ is $L$-pseudomonotone.
	\end{claim}
	
	Clearly, $K$ has values in $P_{wkc}(L^{p'}(T,X^*))$ and it is bounded (see hypotheses $H(A)(iii)$, $H(F)(iii)$).
	
	Next, we consider a sequence $\{x_n\}_{n\geq 1}\subseteq D(L)$ such that
	\begin{eqnarray}\label{eq19}
		\left\{\begin{array}{l}
			x_n\stackrel{w}{\rightarrow}x\in D(L)\ \mbox{in}\ L^p(T,X),L(x_n)\rightarrow L(x)\ \mbox{in}\ L^{p'}(T,X^*),\\
			k^*_n\in K(x_n),k^*_n\stackrel{w}{\rightarrow}k^*\ \mbox{in}\ L^{p'}(T,X^*)\ \mbox{and}\ \limsup\limits_{n\rightarrow\infty}((k^*_n,x_n-x))\leq 0.
		\end{array}\right\}
	\end{eqnarray}
	
	Then we have
	$$k^*_n=h^*_n+f_n\ \mbox{with}\ h^*_n\in a_1(x_n),f_n\in G_1(x_n)\ \mbox{for all}\ n\in\NN.$$
	
	Hypotheses $H(A)(iii)$ and $H(F)(iii)$ imply that
	$$\{h^*_n\}_{n\geq 1}\subseteq L^{p'}(T,X^*)\ \mbox{and}\ \{f_n\}_{n\geq 1}\subseteq L^{p'}(T,H)\ \mbox{is bounded}.$$
	
	So, we may assume (at least for a subsequence) that
	$$h^*_n\stackrel{w}{\rightarrow}h^*\ \mbox{in}\ L^{p'}(T,X^*)\ \mbox{and}\ f_n\stackrel{w}{\rightarrow}f\ \mbox{in}\ L^{p'}(T,H).$$
	
	By (\ref{eq19}) we have
	\begin{eqnarray}\label{eq20}
		&&x_n\stackrel{w}{\rightarrow}x\ \mbox{in}\ W_p(0,b),\nonumber\\
		&\Rightarrow&x_n\rightarrow x\ \mbox{in}\ L^p(T,H)\ (\mbox{see (\ref{eq4})}),\\
		&\Rightarrow&((f_n,x_n-x))=\int^b_0(f_n(t),x_n(t)-x(t))dt\rightarrow 0,\nonumber\\
		&\Rightarrow&\limsup\limits_{n\rightarrow\infty}((h^*_n,x_n-x))\leq 0\ (\mbox{see (\ref{eq19})}),\nonumber\\
		&\Rightarrow&h^*\in a_1(x)\ \mbox{and}\ ((h^*_n,x_n))\rightarrow((h^*,x))\ (\mbox{see Lemma \ref{lem5}})\nonumber.
	\end{eqnarray}
	
	Recall that
	\begin{equation}\label{eq21}
		f_n(t)\in F(t,x_n(t))\ \mbox{for almost all}\ t\in T,\ \mbox{all}\ n\in\NN.
	\end{equation}
	
	By (\ref{eq19}), (\ref{eq20}), (\ref{eq21}) and Proposition 6.6.33 on p. 521 of Papageorgiou and Kyritsi \cite{34}, we have
	\begin{eqnarray*}
		&f(t)&\in\overline{\rm conv}\, w-\limsup\limits_{n\rightarrow\infty}F(t,x_n(t))\\
		&&\subseteq F(t,x(t))\ \mbox{for almost all}\ t\in T\ (\mbox{see hypothesis}\ H(F)(ii)),\\
		\Rightarrow&&f\in G_1(x).
	\end{eqnarray*}
	
	Since $((f_n,x_n-x))=\int^b_0(f_n(t),x_n(t)-x(t))dt\rightarrow 0$, we conclude that $K$ is $L$-pseudomonotone. This proves Claim \ref{cl1}.
	\begin{claim}\label{cl2}
		$K$ is coercive.
	\end{claim}
	
	Let $x\in L^p(T,X)$ and $k^*\in K(x)$. Then
	$$k^*=h^*+f\ \mbox{with}\ h^*\in a_1(x),f\in G_1(x).$$
	
	We have
	\begin{eqnarray}\label{eq22}
		((k^*,x))&=&((h^*,x))+\int^b_0(f(t),x(t))dt\nonumber\\
		&\geq&c_2||x||^p_{L^p(T,X)}-||a_2||_1-\int^b_0|f(t)||x(t)|dt\ (\mbox{see hypothesis}\ H(A)_1(iv)).
	\end{eqnarray}
	
	Note that
	\begin{eqnarray}\label{eq23}
		\int^b_0|f(t)||x(t)|dt&\leq&\int^b_0(a_3(t)|x(t)|+c_3|x(t)|^2)dt\ (\mbox{see hypothesis}\ H(F)_1(iii))\nonumber\\
		&\leq&\int^b_0(c(\epsilon)a_3(t)^2+(c_3+\epsilon)|x(t)|^2)dt\\
		&&(\mbox{using Young's inequality with}\ \epsilon>0).\nonumber
	\end{eqnarray}
	
	Returning to (\ref{eq22}) and using (\ref{eq23}) we see that
	\begin{eqnarray}\label{eq24}
		&&((k^*,x))\geq c_2||x||^p_{L^p(T,X)}-c_4||x||^2_{L^p(T,X)}-c_5\ \mbox{for some}\ c_4,c_5>0\\
		&&(\mbox{recall}\ 2\leq p\ \mbox{and in case}\ p=2,\ \mbox{choose}\ \epsilon>0\ \mbox{small so that}\ c_4<c_2,\nonumber\\
		&&\mbox{see hypothesis}\ H(F)_1(iii)).\nonumber
	\end{eqnarray}
	
	It follows from (\ref{eq24}) that $K$ is coercive. This proves Claim \ref{cl2}.
	
	Now Claims \ref{cl1} and \ref{cl2} permit the use of Proposition \ref{prop3} to find $x\in W_p(0,b)$ solving problem (\ref{eq5}) when $\xi\in X$.
	
	Next, we remove the restriction $\xi\in X$. So, suppose $\xi\in H$. We can find $\{\xi_n\}_{n\geq 1}\subseteq X$ such that $\xi_n\rightarrow\xi$ in $H$ (recall that $X$ is dense in $H$). From the first part of the proof, we know that we can find $x_n\in S(\xi_n)\subseteq W_p(0,b)$ for all $n\in\NN$. We have
	$$\left\{\begin{array}{l}
		-x'_n(t)\in A(t,x_n(t))+F(t,x_n(t))\ \mbox{for almost all}\ t\in T,\\
		x_n(0)=\xi_n,\ n\in\NN .
	\end{array}\right\}$$
	
	It follows that
	\begin{eqnarray}\label{eq25}
		&&-x'_n=h^*_n+f_n\ \mbox{with}\ h^*_n(t)\in A(t,x_n(t)),f_n(t)\in F(t,x_n(t))\\
		&&\hspace{5cm}\mbox{for almost all}\ t\in T,\ \mbox{all}\ n\in\NN.\nonumber
	\end{eqnarray}
	
	We have
	\begin{eqnarray}\label{eq26}
		&&((x'_n,x_n))+((h^*_n,x_n))\leq\int^b_0|f_n(t)||x_n(t)|dt\nonumber\\
		&\Rightarrow&\frac{1}{2}|x_n(b)|^2+c_2||x_n||^p_{L^p(T,X)}\leq c_6+c_7||x_n||^2_{L^p(T,X)}\ \mbox{for some}\ c_6,c_7>0\\
		&&(\mbox{see hypotheses}\ H(A)_1(iv),H(F)_1(iii)\ \mbox{and if}\ p=2\ \mbox{as before we have}\ c_7<c_2).\nonumber
	\end{eqnarray}
	
	From (\ref{eq25}), (\ref{eq26}) and hypotheses $H(A)_1(iii),H(F)_1(iii)$ it follows that
	$$\{x_n\}_{n\geq 1}\subseteq W_p(0,b)\ \mbox{is bounded}.$$
	
	So, we may assume that
	\begin{eqnarray}\label{eq27}
		x_n\stackrel{w}{\rightarrow}x\ \mbox{in}\ W_p(0,b)\ \mbox{and}\ x_n\rightarrow x\ \mbox{in}\ L^p(T,H)\ \mbox{as}\ n\rightarrow\infty\ (\mbox{see (\ref{eq4})}).
	\end{eqnarray}
	
	By (\ref{eq25}) we have for all $n\in\NN$
	\begin{equation}\label{eq28}
		((x'_n,x_n-x))+((h^*_n,x_n-x))=-((f_n,x_n-x))=-\int^b_0(f_n(t),x_n(t)-x(t))dt\,.
	\end{equation}
	
	By Proposition \ref{prop2} we know that
	\begin{eqnarray}\label{eq29}
		&&((x'_n-x',x_n-x))=\frac{1}{2}|x_n(b)-x(b)|^2-\frac{1}{2}|\xi_n-\xi|^2,\nonumber\\
		&\Rightarrow&((x'_n,x-x_n))\leq\frac{1}{2}|\xi_n-\xi|^2+((x',x-x_n)),\nonumber\\
		&\Rightarrow&\limsup\limits_{n\rightarrow\infty}((x'_n,x-x_n))\leq 0\ (\mbox{see (\ref{eq27}) and recall}\ \xi_n\rightarrow\xi\ \mbox{in}\ H).
	\end{eqnarray}
	
	Hypothesis $H(F)_1(iii)$ implies that $\{f_n\}_{n\geq 1}\subseteq L^2(T,H)$ is bounded. Hence
	\begin{eqnarray}\label{eq30}
		&&\int^b_0(f_n(t),x(t)-x_n(t))dt\rightarrow 0\ (\mbox{see (\ref{eq27})}),\nonumber\\
		&\Rightarrow&\limsup\limits_{n\rightarrow\infty}\left[((x'_n,x-x_n))+((f_n,x-x_n))\right]\leq 0\ (\mbox{see (\ref{eq29})})\nonumber\\
		&\Rightarrow&\limsup\limits_{n\rightarrow\infty}((h^*_n,x_n-x))\leq 0\ (\mbox{see (\ref{eq28})}).
	\end{eqnarray}
	
	By hypothesis $H(A)_1(iii)$ we see that
	$$\{h^*_n\}_{n\geq 1}\subseteq L^{p'}(T,X^*)\ \mbox{is bounded}.$$
	
	So, we may assume that
	\begin{equation}\label{eq31}
		h^*_n\stackrel{w}{\rightarrow}h^*\ \mbox{in}\ L^{p'}(T,X^*)\ \mbox{as}\ n\rightarrow\infty .
	\end{equation}
	
	From (\ref{eq27}), (\ref{eq30}), (\ref{eq31}), we see that we can use Lemma \ref{lem5} and infer that
	\begin{equation}\label{eq32}
		h^*(t)\in A(t,x(t))\ \mbox{for almost all}\ t\in T.
	\end{equation}
	
	As we have already mentioned $\{f_n\}_{n\geq 1}\subseteq L^2(T,H)$ is bounded and so we may assume that
	\begin{equation}\label{eq33}
		f_n\stackrel{w}{\rightarrow}f\ \mbox{in}\ L^2(T,H).
	\end{equation}
	
	Using Proposition 6.6.33 on p. 521 of Papageorgiou and Kyritsi \cite{34}, we have
	\begin{eqnarray}\label{eq34}
		f(t)&\in&\overline{\rm conv}\, w-\limsup\limits_{n\rightarrow\infty}F(t,x_n(t))\nonumber\\
		&\subseteq&F(t,x(t))\ \mbox{for almost all}\ t\in T\ (\mbox{see hypothesis}\ H(F)_1(ii)).
	\end{eqnarray}
	
	In (\ref{eq25}) we pass to the limit as $n\rightarrow\infty$ and use (\ref{eq27}), (\ref{eq31}), (\ref{eq33}) to obtain
	\begin{eqnarray*}
		&&-x'=h^*+f\ \mbox{with}\ h^*\in a(x)\ \mbox{(see (\ref{eq32}))},\ f\in G(x)\ (\mbox{see (\ref{eq34})}),\ x(0)=\xi,\\
		&\Rightarrow&x\in S(\xi).
	\end{eqnarray*}
	
	So, we have proved that when $\xi\in H$, the solution set $S(\xi)$ is a nonempty subset of $W_p(0,b)$.
	
	Next, we will prove the compactness of $S(\xi)$ in $W_p(0,b)_w$ and in $C(T,H)$. Let $x\in S(\xi)$. For every $t\in T$ we have
	\begin{eqnarray}\label{eq35}
		&&\int^t_0\left\langle x'(s),x(s)\right\rangle ds+\int^t_0\left\langle h^*(s),x(s)\right\rangle ds\leq\int^t_0|f(s)||x(s)|ds\ \mbox{with}\ h^*\in a(x),\nonumber\\
		&\Rightarrow&\frac{1}{2}|x(t)|^2\leq\frac{1}{2}c_8^2+c_9\int^t_0|x(s)|^2ds\ \mbox{for some}\ c_8,c_9>0,\nonumber\\
		&\Rightarrow&|x(t)|\leq M\ \mbox{for some}\ M>0,\ \mbox{all}\ t\in T,\ \mbox{all}\ x\in S(\xi)\ (\mbox{by Gronwall's inequality}).
	\end{eqnarray}
	
	Then let $r_M:H\rightarrow H$ be the $M$-radial retraction defined by
	$$r_M(x)=\left\{\begin{array}{ll}
		x&\mbox{if}\ |x|\leq M\\
		M\frac{x}{|x|}&\mbox{if}\ M<|x|.
	\end{array}\right.$$
	
	Because of the {\it a priori} bound (\ref{eq35}), we can replace $F(t,x)$ by
	$$\hat{F}(t,x)=F(t,r_M(x)).$$
	
	Note that for all $x\in H,\ t\mapsto\hat{F}(t,x)$ is graph measurable (hence also measurable, see Section 2) and for almost all $t\in T$, $x\mapsto\hat{F}(t,x)$ has a graph which is sequentially closed in $H\times H_w$. Moreover, we see that
	$$|\hat{F}(t,x)|\leq a_4(t)\ \mbox{for almost all}\ t\in T,\ \mbox{all}\ x\in H\ \mbox{with}\ a_4\in L^2(T).$$
	
	We introduce the set
	$$C=\{f\in L^2(T,H):|f(t)|\leq a_4(t)\ \mbox{for almost all}\ t\in T\}.$$
	
	We consider the following Cauchy problem
	\begin{equation}\label{eq36}
		\left\{\begin{array}{l}
			-x'(t)\in A(t,x(t))+f(t)\ \mbox{for almost all}\ t\in T=[0,b],\\
			x(0)=\xi\in H,\ f\in C.
		\end{array}\right\}
	\end{equation}
	
	Let $H:C\rightarrow 2^{C(T,H)}$ be the map (in general, multivalued) that assigns to each $f\in C$  the set of solutions of problem (\ref{eq36}). It is a consequence of Proposition \ref{prop3} and Lemma \ref{lem5}, that $H(\cdot)$ has nonempty values.
	\begin{claim}\label{cl3}
		$H(C)\subseteq C(T,H)$ is compact.
	\end{claim}
	
	Let $\{x_n\}_{n\geq 1}\subseteq H(C)$. Then
	\begin{equation}\label{eq37}
		-x'_n=h^*_n+f_n\ \mbox{with}\ h^*_n\in a(x_n),f_n\in C\ \mbox{for all}\ n\in \NN.
	\end{equation}
	
	Hence for all $t\in T$, we have
	\begin{eqnarray}\label{eq38}
		&&\frac{1}{2}|x_n(t)|^2\leq\frac{1}{2}c^2_{10}+\int^t_0a_4(s)|x_n(s)|ds\ \mbox{for some}\ c_{10}>0,\ \mbox{all}\ n\in\NN\nonumber\\
		&\Rightarrow&|x_n(t)|\leq M_1\ \mbox{for some}\ M_1>0,\ \mbox{all}\ t\in T,\ \mbox{all}\ n\in\NN\\
		&&\hspace{1cm}(\mbox{see Brezis \cite[Lemme A.5, p. 157]{6}}).\nonumber
	\end{eqnarray}
	
	Also, using hypothesis $H(A)_1(iv)$ we have
	\begin{eqnarray}\label{eq39}
		c_2||x_n||^p_{L^p(T,X)}\leq c_{11}+\int^b_0a_4(t)|x_n(t)|dt\leq c_{12}\ \mbox{for all}\ n\in\NN\ (\mbox{see (\ref{eq38})}).
	\end{eqnarray}
	
	It follows from (\ref{eq37}) and (\ref{eq39}) that
	$$\{x_n\}_{n\geq 1}\subseteq W_p(0,b)\ \mbox{is bounded}.$$
	
	So, we may assume that
	\begin{equation}\label{eq40}
		x_n\stackrel{w}{\rightarrow}x\ \mbox{in}\ W_p(0,b),\ h^*_n\stackrel{w}{\rightarrow}h^*\ \mbox{in}\ L^{p'}(T,X^*),\ f_n\stackrel{w}{\rightarrow}f\ \mbox{in}\ L^2(T,H).
	\end{equation}
	
	Passing to the limit as $n\rightarrow\infty$ in (\ref{eq37}) and using (\ref{eq40}), we obtain
	$$-x'=h^*+f.$$
	
	Also, from (\ref{eq37}) we have
	\begin{equation}\label{eq41}
		((h^*_n,x_n-x))+((x'_n,x_n-x))=-\int^b_0(f_n(t),x_n(t)-x(t))dt.
	\end{equation}
	
	Note that
	\begin{eqnarray}\label{eq42}
		\int^b_0(f_n(t),x_n(t)-x(t))dt\rightarrow 0\ (\mbox{see (\ref{eq40}) and (\ref{eq4})}).
	\end{eqnarray}
	
	Also using Proposition \ref{prop2}, we have
	\begin{eqnarray}\label{eq43}
		&&((x'_n-x',x_n-x))=\frac{1}{2}|x_n(b)-x(b)|^2\geq 0\ \mbox{for all}\ n\in\NN\nonumber\\
		&&\hspace{1cm}(\mbox{recall that}\ x_n(0)=x(0)=\xi\ \mbox{for all}\ n\in\NN)\nonumber\\
		&\Rightarrow&((x',x_n-x))\leq((x'_n,x_n-x))\ \mbox{for all}\ n\in\NN .
	\end{eqnarray}
	
	It follows from  (\ref{eq40}) that
	\begin{equation}\label{eq44}
		((x',x_n-x))\rightarrow 0.
	\end{equation}
	
	Returning to (\ref{eq41}), passing to the limit as $n\rightarrow\infty$ and using (\ref{eq42}), (\ref{eq43}), (\ref{eq44}) we obtain
	\begin{eqnarray*}
		&&\limsup\limits_{n\rightarrow\infty}((h^*_n,x_n-x))\leq 0,\\
		&\Rightarrow&h^*\in a(x)\ (\mbox{see Lemma \ref{lem5} and (\ref{eq40})}),\\
		&\Rightarrow&H(C)\ \mbox{is}\ w-\mbox{compact in}\ W_p(0,b).
	\end{eqnarray*}
	
	From the proof of Lemma \ref{lem5} (see (\ref{eq17})), we know that
	\begin{equation}\label{eq45}
		\int^b_0|\left\langle h^*_n(t),x_n(t)-x(t)\right\rangle|dt\rightarrow 0\ \mbox{as}\ n\rightarrow\infty.
	\end{equation}
	
	In a similar fashion we also have
	\begin{equation}\label{eq46}
		\int^b_0|\left\langle h^*(t),x_n(t)-x(t)\right\rangle|dt\rightarrow 0\ \mbox{as}\ n\rightarrow\infty.
	\end{equation}
	
	Also, by (\ref{eq4}), (\ref{eq38}), (\ref{eq40}) and Vitali's theorem, we have
	\begin{equation}\label{eq47}
		\int^b_0|(f_n(t)-f(t),x_n(t)-x(t))|dt\rightarrow 0\ \mbox{as}\ n\rightarrow\infty.
	\end{equation}
	
	For every $t\in T$ and every $n\in\NN$, using Proposition \ref{prop2}, we have
	\begin{eqnarray*}
		&&\frac{1}{2}|x_n(t)-x(t)|^2\leq\int^b_0|\left\langle h^*_n(t)-h^*(t),x_n(t)-x(t)\right\rangle|dt+\\
		&&\hspace{1cm}\int^b_0|(f_n(t)-f(t),x_n(t)-x(t))|dt,\\
		&\Rightarrow&||x_n-x||_{C(T,H)}\rightarrow 0\ \mbox{(see (\ref{eq45}), (\ref{eq46}), (\ref{eq47}))}\\
		&\Rightarrow&H(C)\subseteq C(T,H)\ \mbox{is compact}.
	\end{eqnarray*}
	
	However, from the previous parts of the proof it is clear that $S(\xi)\subseteq H(C)$ is weakly closed in $W_p(0,b)$ and closed in $C(T,H)$. Therefore we conclude that $S(\xi)$ is weakly compact in $W_p(0,b)$ and compact in $C(T,H)$.
\end{proof}

Next, we want to produce a continuous selection of the multifunction $\xi\mapsto S(\xi)$ (we refer to Repov\v{s} and Semenov \cite{repsem} for more details about continuous selections of multivalued mappings). Note that $S(\cdot)$ is in general not convex-valued and so the Michael selection theorem (see Hu and Papageorgiou \cite[Theorem 4.6, p. 92]{23}) cannot be used. To produce a continuous selection of the solution multifunction $\xi\mapsto S(\xi)$, we need to strengthen the conditions on the multimap $A(t,\cdot)$, in order to guarantee that certain Cauchy problems admit a unique solution.

The new hypotheses on the map $A(t,x)$ are the following:

\smallskip
$H(A)_2:$ $A:T\times X\rightarrow 2^{X^*}\backslash\{\emptyset\}$ is a multivalued map such that
\begin{itemize}
	\item[(i)] for every $x\in X,\ t\mapsto A(t,x)$ is graph measurable;
	\item[(ii)] for almost all $t\in T,\ x\mapsto A(t,x)$ is maximal monotone;
	\item[(iii)] for almost all $t\in T$, all $x\in X$ and all $h^*\in A(t,x)$, we have
	$$||h^*||_*\leq a_1(t)+c_1||x||^{p-1}$$
	with $a_1\in L^{p'}(T),c_1>0,2\leq p$;
	\item[(iv)] for almost all $t\in T$, all $x\in X$ and $h^*\in A(t,x)$, we have
	$$\left\langle h^*,x\right\rangle\geq c_2||x||^p-a_2(t),$$
	with $c_2>0,a_2\in L^1(T)_+$.
\end{itemize}
\begin{remark}
	As we have already mentioned in an earlier remark, since now for almost all $t\in T$, $A(t,\cdot)$ is maximal monotone, we do not need the condition on the graph of $A(t,\cdot)$ (see hypothesis $H(A)_1(ii)$ and Barbu \cite[Lemma 1.3, p. 42]{4}).
\end{remark}

	Also, we strengthen the condition on the multifunction $F(t,\cdot)$.

\smallskip
$H(F)_2:$ $F:T\times H\rightarrow P_{f_c}(H)$ is a multifunction such that
\begin{itemize}
	\item[(i)] for every $x\in H$, $t\mapsto F(t,x)$ is graph measurable;
	\item[(ii)] for almost all $t\in T$ and all $x,y\in H$ we have
	$$h(F(t,x),F(t,y))\leq k(t)|x-y|,$$
	with $k\in L^1(T)_+$;
	\item[(iii)] for almost all $t\in T$, all $x\in H$ and all $h\in F(t,x)$, we have
	$$|h|\leq a_3(t)+c_3|x|$$
	with $a_3\in L^2(T)_+,c_3>0$ and if $p=2$, then $\beta^2c_3<c_2$ (see (\ref{eq2})).
\end{itemize}
\begin{remark}
	Hypothesis $H(F)_2(ii)$ is stronger than condition $H(F)_1(ii)$. Indeed, suppose that $H(F)_2(ii)$ holds and we have
	\begin{equation}\label{eq48}
		x_n\rightarrow x,\ h_n\stackrel{w}{\rightarrow}h\ \mbox{in}\ H\ \mbox{and}\ h_n\in F(t,x_n)\ \mbox{for all}\ n\in\NN.
	\end{equation}
	
	By the definition of the Hausdorff metric (see Section 2), we have
	\begin{eqnarray*}
		&&d(h_n,F(t,x))\leq d(h_n,F(t,x_n))+h(F(t,x_n),F(t,x))=h(F(t,x_n),F(t,x)),\\
		&\Rightarrow&d(h_n,F(t,x))\rightarrow 0\ \mbox{as}\ n\rightarrow\infty\ (\mbox{see (\ref{eq48}) and hypothesis}\ H(F)_2(ii))
	\end{eqnarray*}
	
	The function $y\mapsto d(y,F(t,x))$ is continuous and convex, hence weakly lower semicontinuous. Therefore by (\ref{eq48}) we have
	\begin{eqnarray*}
		&&d(h,F(t,x))\leq\liminf\limits_{n\rightarrow\infty}d(h_n,F(t,x))=0\\
		&\Rightarrow&h\in F(t,x).
	\end{eqnarray*}

This proves that condition $H(F)_1(ii)$ holds.
\end{remark}

So, we can use Theorem \ref{th6} and establish that given any $\xi\in H$, the solution set $S(\xi)$ is nonempty, weakly compact in $W_p(0,b)$ and compact in $C(T,H)$. The next result extends an earlier result  of Cellina and Ornelas \cite{9} for differential inclusions in $\RR^N$ with $A\equiv 0$.
\begin{prop}\label{prop7}
	If hypotheses $H(A)_2,H(F)_2$ hold, then there exists a continuous map $\vartheta:H\rightarrow C(T,H)$ such that
	$$\vartheta(\xi)\in S(\xi)\ \mbox{for all}\ \xi\in H.$$
\end{prop}
\begin{proof}
	Consider the following auxiliary Cauchy problem
	$$-x'(t)\in A(t,x(t))\ \mbox{for almost all}\ t\in T=[0,b],\ x(0)=\xi.$$
	
	This problem has a unique solution $x_0(\xi)\in W_p(0,b)$ (see Proposition \ref{prop3} and use the monotonicity of $A(t,\cdot)$ and Proposition \ref{prop2} to check the uniqueness of this solution).
	
	If $\xi_1,\xi_2\in H$, then
	$$-x'_0(\xi_1)=h^*_1\ \mbox{and}\ -x'_0(\xi_2)=h^*_2\ \mbox{with}\ h^*_k\in a(x_0(\xi_k))\ \mbox{for}\ k=1,2.$$
	
	So, using Proposition \ref{prop2}, we have
	\begin{eqnarray}\label{eq49}
		&&\frac{1}{2}|x_0(\xi_1)(t)-x_0(\xi_2)(t)|^2+\int^t_0\left\langle h^*_1(s)-h^*_2(s),x_0(\xi_1)(s)-x_0(\xi_2)(s)\right\rangle ds=\nonumber\\
		&&\hspace{1cm}\frac{1}{2}|\xi_1-\xi_2|^2\ \mbox{for all}\ t\in T,\nonumber\\
		&\Rightarrow&||x_0(\xi_1)-x_0(\xi_2)||_{C(T,H)}\leq|\xi_1-\xi_2|\ (\mbox{recall}\ A(t,\cdot)\ \mbox{is monotone}).
	\end{eqnarray}
	
	We consider the multifunction $\Gamma_0:H\rightarrow P_{wkc}(L^1(T,H))$ defined by
	$$\Gamma_0(\xi)=S^1_{F(\cdot,x_0(\xi)(\cdot))}\ \mbox{for all}\ \xi\in H.$$
	
	We have
	\begin{eqnarray*}
		h(\Gamma_0(\xi),\Gamma_0(\xi_2))&&=\sup\left[|\sigma(g,\Gamma_0(\xi_1))-\sigma(g,\Gamma_0(\xi_2))|:g\in L^{\infty}(T,H)=L^1(T,H)^*,\right.\\
		&&\hspace{1cm}\left.||g||_{L^{\infty}(T,H)}\leq 1\right]\\
		&&\leq\int^b_0\sup\left[|\sigma(v,F(t,x_0(\xi_1)(t)))-\sigma(v,F(t,x_0(\xi_2)(z)))|:|v|\leq 1\right]dt\\
		&&(\mbox{see Papageorgiou and Kyritsi \cite[Theorem 6.4.16, p. 492]{34}})\\
		&&=\int^b_0h(F(t,x_0(\xi_1)(t)),F(t,x_0(\xi_2)(t)))dt\\
		&&\leq\int^b_0k(t)|x_0(\xi_1)(t)-x_0(\xi_2)(t)|dt\\
		&&\leq\int^b_0k(t)|\xi_1-\xi_2|dt\ (\mbox{see (\ref{eq49})})\\
		&&=||k||_1|\xi_1-\xi_2|,\\
		\Rightarrow \xi\mapsto\Gamma_0(\xi)&&\mbox{is $h$-Lipschitz}.
	\end{eqnarray*}
	
	Also, $\Gamma_0(\cdot)$ has decomposable values. So, we can apply the selection theorem of Bressan and Colombo \cite{5} (see also Hu and Papageorgiou \cite[Theorem 8.7, p. 245]{23}) and find a continuous map $\gamma_0:H\rightarrow L^1(T,H)$  such that $\gamma_0(\xi)\in\Gamma_0(\xi)$ for all $\xi\in H$. Evidently $\gamma_0(\xi)\in L^2(T,H)$ for all $\xi\in H$.
	
	We consider the following auxiliary Cauchy problem:
	$$-x'(t)\in A(t,x(t))+\gamma_0(\xi)(t)\ \mbox{for almost all}\ t\in T,\ x(0)=\xi.$$
	
	This problem has a unique solution $x_1(\xi)\in W_p(0,b)$.
	
	By induction we will produce two sequences
	$$\{x_n(\xi)\}_{n\geq 1}\subseteq W_p(0,b)\ \mbox{and}\ \{\gamma_n(\xi)\}_{n\geq 1}\subseteq L^2(T,H)$$
	which satisfy:
	\begin{itemize}
		\item[(a)] $x_n(\xi)\in W_p(0,b)$ is the unique solution of the Cauchy problem
		\begin{equation}\label{eq50}
			-x'(t)\in A(t,x(t))+\gamma_{n-1}(\xi)(t)\ \mbox{for almost all}\ t\in T,\ x(0)=\xi ;
		\end{equation}
		\item[(b)] $\xi\mapsto\gamma_n(\xi)$ is continuous from $H$ into $C(T,H)$;
		\item[(c)] $\gamma_n(\xi)(t)\in F(t,x_n(\xi)(t))$ for almost all $t\in T$, all $\xi\in H$;
		\item[(d)] $|\gamma_n(\xi)(t)-\gamma_{n-1}(\xi)(t)|\leq k(t)\beta_n(\xi)(t)$ for almost all $t\in T$, all $\xi\in H$, where
		 $$\beta_n(\xi)(t)=2\int^t_0\eta(\xi)(s)\frac{(\tau(t)-\tau(s))^{n-1}}{(n-1)!}ds+2b\left(\overset{n}{\underset{\mathrm{k=0}}\sum}\frac{\epsilon}{2^{k+1}}\right)\frac{\tau(t)^{n-1}}{(n-1)!}$$
		with $\epsilon>0,\eta(\xi)(t)=a_2(t)+c_2|x_0(\xi)(t)|,\tau(t)=\int^t_0k(s)ds.$
	\end{itemize}
	
	Note that the maps $\xi\mapsto\eta(\xi)$ and $\xi\mapsto\beta_n(\xi)$ are continuous from $H$ into $L^1(T)$.
	
	So, suppose we have produced $\{x_k(\xi)\}^n_{k=1}$ and $\{\gamma_k(\xi)\}^n_{k=1}$ (induction hypothesis). Let $x_{n+1}(\xi)\in W_p(0,b)$ be the unique solution of the Cauchy problem
	\begin{equation}\label{eq51}
		-x'(t)\in A(t,x(t))+\gamma_n(\xi)(t)\ \mbox{for almost all}\ t\in T,\ x(0)=\xi.
	\end{equation}
	
	By (\ref{eq50}) and (\ref{eq51}) we have
	\begin{eqnarray}
		&&-x'_n(\xi)=h^*_n+\gamma_{n-1}(\xi)\ \mbox{and}\ -x'_{n+1}(\xi)=h^*_{n+1}+\gamma_n(\xi)\ \mbox{in}\ L^{p'}(T,X^*)\label{eq52}\\
		&&\mbox{with}\ h^*_n\in a(x_n(\xi)),h^*_{n+1}\in a(x_{n+1}(\xi)).\label{eq53}
	\end{eqnarray}
	
	Using (\ref{eq52}) we can write
	\begin{eqnarray}\label{eq54}
		&&\left\langle x'_{n+1}(\xi)(t)-x'_n(\xi)(t),x_{n+1}(\xi)(t)-x_n(\xi)(t)\right\rangle\nonumber\\
		&=&\left\langle h^*_n(\xi)(t)-h^*_{n+1}(\xi)(t),x_{n+1}(\xi)(t)-x_n(\xi)(t)\right\rangle+(\gamma_{n-1}(\xi)(t)-\nonumber\\
		&&\hspace{1cm}\gamma_n(\xi)(t),x_{n+1}(\xi)(t)-x_n(\xi)(t))\nonumber\\
		&\leq&(\gamma_{n-1}(\xi)(t)-\gamma_n(\xi)(t),x_{n+1}(\xi)(t)-x_n(\xi)(t))\ \mbox{for almost all}\ t\in T\nonumber\\
		&&\hspace{1cm}(\mbox{see hypothesis}\ H(A)_2(ii)\ \mbox{and}\ \eqref{eq53})\nonumber\\
		&\Rightarrow&\frac{1}{2}\ \frac{d}{dt}|x_{n+1}(\xi)(t)-x_n(\xi)(t)|^2\leq|\gamma_{n-1}(\xi)(t)-\gamma_n(\xi)(t)|\ |x_{n+1}(\xi)(t)-x_n(\xi)(t)|\nonumber\\
		&&\hspace{1cm}\mbox{for almost all}\ t\in T\ (\mbox{see Proposition \ref{prop2}}),\nonumber\\
		&\Rightarrow&\frac{1}{2}|x_{n+1}(\xi)(t)-x_n(\xi)(t)|^2\leq\int^t_0|\gamma_{n-1}(\xi)(s)-\gamma_n(\xi)(s)|\ |x_{n+1}(\xi)(s)-x_n(\xi)(s)|ds\\
		&&\hspace{1cm}\mbox{for all}\ t\in T.\nonumber
	\end{eqnarray}
	
	By (\ref{eq54}) and Lemma A.5, p. 157 of Brezis \cite{6}, we infer that
	\begin{eqnarray}\label{eq55}
		|x_{n+1}(\xi)(t)-x_n(\xi)(t)|&\leq&\int^t_0|\gamma_{n-1}(\xi)(s)-\gamma_n(\xi)(s)|ds\nonumber\\
		&\leq&\int^t_0k(s)\beta_n(\xi)(s)ds\ (\mbox{by the induction hypothesis, see (d)})\nonumber\\
		&=&2\int^t_0k(s)\int^s_0\eta(\xi)(r)\frac{(\tau(s)-\tau(r))^{n-1}}{(n-1)!}dr\ ds+\nonumber\\
		 &&\hspace{1cm}2b\left(\overset{n}{\underset{\mathrm{k=0}}\sum}\frac{\epsilon}{2^{k+1}}\right)\int^t_0k(s)\frac{\tau(s)^{n-1}}{(n-1)!}ds\nonumber\\
		&=&2\int^t_0\eta(\xi)(s)\int^t_s k(r)\frac{(\tau(r)-\tau(s))^{n-1}}{(n-1)!}dr\ ds+\nonumber\\
		 &&\hspace{1cm}2b\left(\overset{n}{\underset{\mathrm{k=0}}\sum}\frac{\epsilon}{2^{k+1}}\right)\int^t_0\frac{d}{ds}\frac{\tau(s)^{n}}{n!}ds\nonumber\\
		&=&2\int^t_0\eta(\xi)(s)\int^t_s \frac{d}{dr}\frac{(\tau(r)-\tau(s))^{n}}{n!}dr\ ds+2b\left(\overset{n}{\underset{\mathrm{k=0}}\sum}\frac{\epsilon}{2^{k+1}}\right)\frac{\tau(t)^{n}}{n!}\nonumber\\
		 &=&2\int^t_0\eta(\xi)(s)\frac{(\tau(t)-\tau(s))^{n}}{n!}ds+2b\left(\overset{n}{\underset{\mathrm{k=0}}\sum}\frac{\epsilon}{2^{k+1}}\right)\frac{\tau(t)^{n}}{n!}\nonumber\\
		&<&\beta_{n+1}(\xi)(t)\ \mbox{for almost all}\ t\in T\ (\mbox{see (d)}).
	\end{eqnarray}
	
	Using the induction hypothesis (see (c)) and hypothesis $H(F)_2(ii)$, we have
	\begin{eqnarray}\label{eq56}
		d(\gamma_n(\xi)(t),F(t,x_{n+1}(\xi)(t)))&\leq&h(F(t,x_n(\xi)(t)),F(t,x_{n+1}(\xi)(t)))\nonumber\\
		&\leq&k(t)|x_n(\xi)(t)-x_{n+1}(\xi)(t)|\nonumber\\
		&<&k(t)\beta_{n+1}(\xi)(t)\ \mbox{for almost all}\ t\in T\ (\mbox{see (\ref{eq55})}).
	\end{eqnarray}
	
	Consider the multifunction $\Gamma_{n+1}:H\rightarrow 2^{L^1(T,H)}$ defined by
	$$\Gamma_{n+1}(\xi)=\{f\in S^1_{F(\cdot,x_{n+1}(\xi)(\cdot))}:|\gamma_n(\xi)(t)-f(t)|<k(t)\beta_{n+1}(\xi)(t)\ \mbox{for almost all}\ t\in T\}.$$
	
	By (\ref{eq56}) and Lemma 8.3 on p. 239 of Hu and Papageorgiou \cite{23}, we have that
	\begin{eqnarray*}
		&&\xi\mapsto\Gamma_{n+1}(\xi)\ \mbox{has nonempty decomposable values and it is lsc,}\\
		&\Rightarrow&\xi\mapsto\overline{\Gamma_{n+1}(\xi)}\ \mbox{is lsc with decomposable values.}
	\end{eqnarray*}
	
	We can apply the selection theorem of Bressan and Colombo \cite[Theorem 3]{5}  to find a continuous map $\gamma_{n+1}:H\rightarrow L^1(T,H)$ such that $\gamma_{n+1}(\xi)\in\overline{\Gamma_{n+1}(\xi)}$ for all $\xi\in H$.
	
	This completes the induction process and we have produced  two sequences $\{x_n(\xi)\}_{n\geq 1}$, $\{\gamma_n(\xi)\}_{n\geq 1}$ which satisfy properties $(a)\rightarrow(d)$ stated earlier.
	
	From (\ref{eq55}) we have
	\begin{eqnarray}\label{eq57}
		\int^b_0|\gamma_n(\xi)(t)-\gamma_{n-1}(\xi)(t)|dt&<&\int^b_0\beta_{n+1}(\xi)(t)dt\nonumber\\
		&<&\int^b_0\eta(\xi)(t)\frac{(\tau (b)-\tau(t))^n}{n!}dt+2b\epsilon\frac{\tau(b)^n}{n!}\nonumber\\
		&\leq&\frac{\tau(b)^n}{n!}[||\eta(\xi)||_1+2b\epsilon].
	\end{eqnarray}
	
	Recall that $\xi\mapsto\eta(\xi)$ is continuous from $H$ into $L^1(H)$ and maps bounded sets to bounded sets. So,  it follows from (\ref{eq57}) that
	$$\{\gamma_n(\xi)\}_{n\geq 1}\subseteq L^1(T)\ \mbox{is Cauchy, uniformly on bounded sets of}\ H.$$
	
	Moreover, from (\ref{eq55}) and (\ref{eq57}), we have
	\begin{eqnarray*}
		 &&||x_{n+1}(\xi)-x_n(\xi)||_{C(T,H)}\leq||\gamma_n(\xi)-\gamma_{n-1}(\xi)||_{L^1(T,H)}\leq\frac{\tau(b)^n}{n!}[||\eta(\xi)||_1+2b\epsilon],\nonumber\\
		&\Rightarrow&\{x_n(\xi)\}_{n\geq 1}\subseteq C(T,H)\ \mbox{is Cauchy, uniformly on bounded sets}.
	\end{eqnarray*}
	
	Therefore we have
	\begin{equation}\label{eq58}
		x_n(\xi)\rightarrow x(\xi)\ \mbox{in}\ C(T,H)\ \mbox{and}\ \gamma_n(\xi)\rightarrow\gamma(\xi)\ \mbox{in}\ L^1(T,H).
	\end{equation}
	
	Evidently, $\xi\mapsto x(\xi)$ is continuous from $H$ into $C(T,H)$, while because of hypothesis $H(F)_2(iii)$, we have $\gamma_n(\xi)\rightarrow\gamma(\xi)$ in  $L^2(T,H)$. Let $\hat{x}(\xi)\in W_p(0,b)$ be the unique solution of
	$$-y'(t)\in A(t,y(t))+\gamma(\xi)(t)\ \mbox{for almost all}\ t\in T,\ y(0)=0.$$
	
	As before, exploiting the monotonicity of $A(t,\cdot)$ (see hypothesis $H(A)_2(ii)$), we have
	\begin{eqnarray*}
		&&\frac{1}{2}|x_{n+1}(\xi)(t)-\hat{x}(\xi)(t)|^2\leq\int^t_0|\gamma_n(\xi)(s)-\gamma(\xi)(s)|\ |x_{n+1}(\xi)(s)-\hat{x}(\xi)(s)|ds\\
		&&\hspace{1cm}\mbox{for all}\ t\in T,\\
		&\Rightarrow&||x_{n+1}(\xi)-\hat{x}(\xi)||_{C(T,H)}\leq||\gamma_n(\xi)-\gamma(\xi)||_{L^1(T,H)}\\
		&&\hspace{1cm}(\mbox{see Brezis \cite[Lemma A.5, p. 157]{6}})\\
		&\Rightarrow&x(\xi)=\hat{x}(\xi)\ (\mbox{see (\ref{eq58})}).
	\end{eqnarray*}
	
	So, $x(\xi)\in S(\xi)$ and the map $\vartheta:H\rightarrow C(T,H)$ defined by $\vartheta(\xi)=x(\xi)$ is a continuous selection of the solution multifunction $\xi\mapsto S(\xi)$.
\end{proof}

An easy, but useful consequence of Proposition \ref{prop7} and of its proof, is a parametric version of the Filippov-Gronwall inequality (see Aubin and Cellina \cite[Theorem 1, pp. 120-121]{1} and Frankowska \cite{19}) for differential inclusions.

So, we consider the following parametric version of problem (\ref{eq5}):
$$-x'(t)\in A(t,x(t))+F(t,x(t),\lambda)\ \mbox{for almost all}\ t\in T,\ x(0)=\xi(\lambda).$$

The parameter space $D$ is a complete metric space. The hypotheses on the parametric vector field $F(t,x,\lambda)$ and the initial condition $\xi(\lambda)$ are the following:

\smallskip
$H(F)'_2:$ $F:T\times H\times D\rightarrow P_{f_c}(H)$ is a multifunction such that
\begin{itemize}
	\item[(i)] for all $(x,\lambda)\in H\times D,\ t\mapsto F(t,x,\lambda)$ is graph measurable;
	\item[(ii)] for almost all $t\in T$, all $x,y\in H$, all $\lambda\in D$, we have
	$$h(F(t,x,\lambda),F(t,y,\lambda))\leq k(t)|x-y|$$
	with $k\in L^1(T)_+$;
	\item[(iii)] for almost all $t\in T$, all $x\in H$, all $\lambda\in D$ and all $h\in F(t,x)$, we have
	$$|h|\leq a_3(t)+c_3|x|$$
	with $a_3\in L^2(T)_+,\ c_3>0$ and if $p=2$, then $\beta^2c_3<c_2$ (see (\ref{eq2}));
	\item[(iv)] for almost all $t\in T$, all $x\in H$, the multifunction $\lambda\mapsto F(t,x,\lambda)$ is lsc.
\end{itemize}

\smallskip
$H_0:$ the mapping $\lambda\mapsto\xi(\lambda)$ is continuous from $D$ into $H$.

\smallskip
Assume that $\lambda\mapsto (u(\lambda),h(\lambda))$ is a continuous map from $D$ into $C(T,H)\times L^2(T,H)$. We can find a continuous map $p:D\rightarrow L^2(T)$ such that
$$d(h(\lambda)(t),F(t,u(\lambda)(t),\lambda))\leq p(\lambda)(t)\ \mbox{for almost all}\ t\in T$$
(see hypothesis $H(F)'_2(iii)$).

In what follows, by $e(h,\lambda)\in W_p(0,b)$ we denote the unique solution of the Cauchy problem
$$-u'(t)\in A(t,u(t))+h(t)\ \mbox{for almost all}\ t\in T,\ u(0)=\xi(\lambda),$$
with $h\in L^2(T,H)$.

We have the following approximation result.
\begin{prop}\label{prop8}
	Assume that hypotheses $H(A)_2,H(F)'_2,H_0$ hold, $\lambda\mapsto (u(\lambda),h(\lambda))$ is a continuous map from $D$ into $C(T,H)\times L^2(T,H)$ with $u(\lambda)=e(h(\lambda),\lambda)$, $\epsilon>0$ and $p:D\rightarrow L^2(T)_+$ is a continuous map such that
	$$d(h(\lambda)(t),F(t,u(\lambda)(t),\lambda))\leq p(\lambda)(t)\ \mbox{for almost all}\ t\in T.$$
	Then there exists a continuous map $\lambda\mapsto (x(\lambda),f(\lambda))$ from $D$ into $C(T,H)\times L^2(T,H)$ such that
	$$x(\lambda)=e(f(\lambda),\lambda)\ \mbox{with}\ f(\lambda)\in S^2_{F(\cdot,x(\lambda)(\cdot),\lambda)}$$
	and $|x(\lambda)(t)-u(\lambda)(t)|\leq b\epsilon e^{\tau(t)}+\int^t_0p(\lambda)(s)e^{\tau(t)-\tau(s)}ds$ for all $t\in T$ with $\tau(t)=\int^t_0k(s)ds$.
\end{prop}
\begin{proof}
	Consider the multifunction $R_{\epsilon}:D\rightarrow 2^{L^1(T,H)}$ defined by
	$$R_{\epsilon}(\lambda)=\{v\in S^1_{F(\cdot,u(\lambda)(\cdot),\lambda)}:|v(t)-h(\lambda)(t)|<p(\lambda)(t)+\epsilon\ \mbox{for almost all}\ t\in T\}.$$
	
	This multifunction has nonempty, decomposable values and it is lsc (see Hu and Papageorgiou \cite[Lemma 8.3, p. 239]{23}). Hence $\lambda\mapsto\overline{R_{\epsilon}(\lambda)}$ has the same properties. So, we can find a continuous map $\gamma_0:D\rightarrow L^1(T,H)$ such that
	$$\gamma_0(\lambda)\in\overline{R_{\epsilon}(\lambda)}\ \mbox{for all}\ \lambda\in D.$$
	
	Let $x_1(\lambda)\in W_p(0,b)$ be the unique solution of the following Cauchy problem
	$$-x'(t)\in A(t,x(t))+\gamma_0(\lambda)(t)\ \mbox{for almost all}\ t\in T,\ x(0)=\xi(\lambda).$$
	
	Then as in the proof of Proposition \ref{prop7},  we can generate by induction two sequences
	$$\{x_n(\lambda)\}_{n\geq 1}\subseteq W_p(0,b)\ \mbox{and}\ \{\gamma_n(\lambda)\}_{n\geq 1}\subseteq L^2(T,H)$$
	satisfying properties $(a)\rightarrow(d)$ listed in the proof of Proposition \ref{prop7}.
	
	As before (see the proof of Proposition \ref{prop7}), we have
	$$|x_{n+1}(\lambda)(t)-x_n(\lambda)(t)|\leq||\gamma_n(\lambda)-\gamma_{n-1}(\lambda)||_{L^1(T,H)}\ \mbox{for all}\ (\lambda,t)\in D\times T.$$
	
	From this inequality and property $(d)$ of the sequences (see the proof of Proposition \ref{prop7}), we infer that
	$$\{x_n(\lambda)\}_{n\geq 1}\subseteq C(T,H)\ \mbox{and}\ \{\gamma_n(\lambda)\}_{n\geq 1}\subseteq L^1(T,H)$$
	are both Cauchy uniformly in $\lambda\in K\subseteq D$ compact (recall that $\lambda\mapsto p(\lambda)$ is continuous, hence locally bounded). So, we have
	$$x_n(\lambda)\rightarrow\hat{x}(\lambda)\ \mbox{in}\ C(T,H)\ \mbox{and}\ \gamma_n(\lambda)\rightarrow\hat{\gamma}(\lambda)\ \mbox{in}\ L^1(T,H)$$
	and both maps $D\ni \lambda\mapsto\hat{x}(\lambda)\in C(T,H)$ and $D\ni\lambda\mapsto\hat{\gamma}(\lambda)\in L^1(T,H)$ are continuous. Moreover, we have $\hat{\gamma}(\lambda)\in S^2_{F(\cdot,\hat{x}(\lambda)(\cdot),\lambda)}$ (see the proof of Theorem \ref{th6}) and that $\lambda\mapsto\hat{\gamma}(\lambda)$ is continuous from $D$ into $L^2(T,H)$. If $x(\lambda)=e(\hat{\gamma}(\lambda),\lambda)$, then
	\begin{eqnarray*}
		&&|x_n(\lambda)(t)-x(\lambda)(t)|\leq\int^b_0|\gamma_{n-1}(\lambda)(s)-\hat{\gamma}(\lambda)(s)|ds\rightarrow 0\ \mbox{for all}\ t\in T,\\
		&\Rightarrow&\hat{x}(\lambda)=x(\lambda)=e(\hat{\gamma}(\lambda),\lambda)\ \mbox{for all}\ \lambda\in D.
	\end{eqnarray*}
	
	From the triangle inequality, we have
	\begin{eqnarray*}
		&&|u(\lambda)(t)-x_n(\lambda)(t)|\\
		 &\leq&|u(\lambda)(t)-x_1(\lambda)(t)|+\overset{n-1}{\underset{\mathrm{k=1}}\sum}|x_k(\lambda)(t)-x_{k+1}(\lambda)(t)|\ \mbox{for all}\ t\in T.
	\end{eqnarray*}
	
	Using property $(d)$ (see the proof of Proposition \ref{prop7}), we have
	 $$|x_k(\lambda)(t)-x_{k+1}(\lambda)(t)|\leq\frac{1}{k!}\int^t_0p(\lambda)(s)(\tau(t)-\tau(s))^kds+\frac{b\epsilon}{k!}\tau(t)^k\ \mbox{for all}\ t\in T.$$
	
	So, finally we can write that
	$$|u(\lambda)(t)-x(\lambda)(t)|\leq b\epsilon e^{\tau(b)}+\int^t_0p(\lambda)(s)e^{\tau(t)-\tau(s)}ds\ \mbox{for all}\ t\in T,\ \mbox{all}\ \lambda\in D.$$
\end{proof}

We want to strengthen Proposition \ref{prop7}, and require that the selection $\vartheta(\cdot)$ passes through a preassigned solution. We mention that an analogous result for differential inclusions in $\RR^N$ with $A\equiv 0$, was proved by Cellina and Staicu \cite{10}.

We start with a simple technical lemma.
\begin{lemma}\label{lem9}
	If $\{u_k\}_{k=0}^N\subseteq L^1(T,H)$ and $\{T_k(\xi)\}^N_{k=0}$ is a partition of $T=[0,b]$ with endpoints which depend continuously on $\xi\in H$, then there exists $\hat{d}\in L^1(T)_+$ for which the following holds:
	\begin{center}
	``given $\epsilon>0$, we can find $\delta>0$ such that\\
	 $|\xi-\xi'|\leq\delta\Rightarrow\left|\overset{N}{\underset{\mathrm{k=0}}\sum}\chi_{T_k(\xi)}(t)u_k(t)-\overset{N}{\underset{\mathrm{k=0}}\sum}\chi_{T_k(\xi')}(t)u_k(t)\right|\leq\hat{d}(t)\chi_C(t)$\\
	with $C\subseteq T$ measurable and $|C|_1\leq\epsilon$".
	\end{center}
\end{lemma}
\begin{proof}
	We have
	\begin{eqnarray}\label{eq59}
		 \left|\overset{N}{\underset{\mathrm{k=0}}\sum}\chi_{T_k(\xi)}(t)u_k(t)-
\overset{N}{\underset{\mathrm{k=0}}\sum}\chi_{T_k(\xi')(t)}(t)u_k(t)\right|&\leq&\overset{N}{\underset{\mathrm{k=0}}\sum}\left|\chi_{T_k(\xi)}(t)-\chi_{T_k(\xi')}\right||u_k(t)|\nonumber\\
		&=&{\underset{\mathrm{k}}\sum}\chi_{T_k(\xi)\Delta T_k(\xi')}(t)|u_k(t)|.
	\end{eqnarray}
	
	We set $\hat{d}(t)=\overset{N}{\underset{\mathrm{k=0}}\sum}|u_k(t)|\in L^1(T)_+$. From the hypothesis concerning the partition $\{T_k(\xi)\}^N_{k=0}$ of $T$, we see that given $\epsilon>0$, we can find $\delta>0$ such that
	\begin{equation}\label{eq60}
		|\xi-\xi'|\leq\delta\Rightarrow\chi_{T_k(\xi)\Delta T_k(\xi')}(t)\leq\chi_C(t)\ \mbox{for almost all}\ t\in T,\ \mbox{all}\ k\in\{0,\ldots,N\}
	\end{equation}
	with $C\subseteq T$ measurable, $|C|_1\leq\epsilon$. Then
	\begin{eqnarray*}
		 \left|\overset{N}{\underset{\mathrm{k=0}}\sum}\chi_{T_k(\xi)}(t)u_k(t)-\overset{N}{\underset{\mathrm{k=0}}\sum}\chi_{T_k(\xi')}(t)u_k(t)\right|&\leq&\chi_C(t)\overset{N}{\underset{\mathrm{k=0}}\sum}|u_k(t)|\ \mbox{(see (\ref{eq59}),(\ref{eq60}))}\\
		&=&\hat{d}(t)\chi_C(t)\ \mbox{for almost all}\ t\in T.
	\end{eqnarray*}
The proof is now complete.
\end{proof}

With this lemma, we can produce a continuous selection of the solution multifunction $\xi\mapsto S(\xi)$, which passes through a preassigned point.

\begin{prop}\label{prop10}
	If hypotheses $H(A)_2,H(F)_2$ hold, $K\subseteq H$ is compact, $\xi_0\in K$ and $v\in S(\xi_0)$, then there exists a continuous map $\psi:K\rightarrow C(T,H)$ such that
	$$\psi(\xi)\in S(\xi)\ \mbox{for all}\ \xi\in K\ \mbox{and}\ \psi(\xi_0)=v.$$
\end{prop}
\begin{proof}
	Since $v\in S(\xi_0)$, we have
	\begin{eqnarray}\label{eq61}
		-v'(t)\in A(t,v(t))+f(t)\ \mbox{for almost all}\ t\in T,\ v(0)=\xi_0,
	\end{eqnarray}
	with $f\in S^2_{F(\cdot,v(\cdot))}$. Given $g\in L^2(T,H)$, we consider the unique solution of the Cauchy problem
	\begin{equation}\label{eq62}
		-y'(t)\in A(t,y(t))+g(t)\ \mbox{for almost all}\ t\in T,\ y(0)=\xi\in H.
	\end{equation}
	
	In what follows, by $e(g,\xi)\in W_p(0,b)$ we denote the unique solution of problem (\ref{eq62}) and we set $\mu_0(\xi)=e(f,\xi)$. An easy application of the Yankov-von Neumann-Aumann selection theorem (see Hu and Papageorgiou \cite[Theorem 2.14, p. 158]{23}), gives $\gamma_0(\xi)\in L^2(T,H)$ such that
	\begin{eqnarray*}
		&\gamma_0(\xi)(t)\in F(t,\mu_0(\xi)(t))&\mbox{for almost all}\ t\in T\\
		\quad\qquad\mbox{\quad and}&\qquad |f(t)-\gamma_0(\xi)(t)|&=d(f(t),F(t,\mu_0(\xi)(t)))\\
		&&\leq k(t)|v(t)-\mu_0(\xi)(t)|\ (\mbox{see hypothesis}\ H(F)_2(ii))\\
		&&=k(t)|e(f,\xi_0)(t)-e(f,\xi)(t)|\\
		&&\leq k(t)|\xi_0-\xi|\ \mbox{for almost all}\ t\in T\ (\mbox{see (\ref{eq49})}).
	\end{eqnarray*}
	
	Let $\vartheta>0$ we define
	$$\delta(\xi)=\left\{\begin{array}{ll}
		\min\left\{2^{-3}\vartheta,\frac{|\xi-\xi_0|}{2}\right\}&\mbox{if}\ \xi\neq\xi_0\\
		2^{-3}\vartheta&\mbox{if}\ \xi=\xi_0.
	\end{array}\right.$$
	
	The family $\{B_{\delta(\xi)}(\xi)\}_{\xi\in K}$ is an open cover of the compact set $K$. So, we can find $\{\xi_k\}^N_{k=0}\subseteq K$ such that $\{B_{\delta(\xi_k)}(\xi_k)\}^N_{k=0}$ is a finite subcover of $K$. Let $\{\eta_k\}^N_{k=0}$ be a locally Lipschitz partition of unity subordinated to the finite subcover. We define
	\begin{eqnarray*}
		&&T_0(\xi)=[0,\eta_0(\xi)b]\ \mbox{and}\ T_k(\xi)=\left[\left(\overset{k-1}{\underset{\mathrm{i=0}}\sum}\eta_i(\xi)\right)b,\left(\overset{k}{\underset{\mathrm{i=0}}\sum}\eta_i(\xi)\right)(b),\right]\\
		&&\hspace{1cm}\mbox{for all}\ k\in\{1,\ldots,N\}.
	\end{eqnarray*}
	
	The endpoints in these intervals are continuous functions of $\xi$. We consider the following Cauchy problem
	\begin{equation}\label{eq63}
		-y'(t)\in A(t,y(t))+\overset{N}{\underset{\mathrm{k=0}}\sum}\chi_{T_k(\xi)}(t)\gamma_0(\xi_k)(t)\ \mbox{for almost all}\ t\in T,\ y(0)=\xi\in K.
	\end{equation}
	
	Problem (\ref{eq63}) has a unique solution $\mu_1(\xi)\in W_p(0,b)$. Let
	$$\lambda_0(\xi)(\cdot)=\overset{N}{\underset{\mathrm{k=0}}\sum}\chi_{T_k(\xi)}(\cdot)\gamma_0(\xi_k)(\cdot)\in L^2(T,H).$$
	
	Using Lemma \ref{lem9}, we can find $\hat{d}\in L^1(T)_+$ such that, for any given $\epsilon>0$, we can find $\delta>0$ for which we have
	\begin{equation}\label{eq64}
		\xi,\xi'\in K,|\xi-\xi'|_1\leq\delta\Rightarrow|\lambda_0(\xi)(t)-\lambda_0(\xi')(t)|\leq\hat{d}(t)\chi_{C}(t)\ \mbox{for almost all}\ t\in T,
	\end{equation}
	with $C\subseteq T$ measurable, $|C|_1\leq\epsilon$. We have $\mu_1(\xi')=e(\lambda_0(\xi'),\xi')$. As before, exploiting the monotonicity of $A(t,\cdot)$ (see hypothesis $H(A)_2(ii)$) and using Lemma A.5, p. 157, of Brezis \cite{6}, we have
	\begin{equation}\label{eq65}
		|\mu_1(\xi)(t)-\mu_1(\xi')(t)|\leq|\xi-\xi'|+\int^t_0|\lambda_0(\xi)(s)-\lambda_0(\xi')(s)|ds\ \mbox{for all}\ t\in T.
	\end{equation}
	
	Let $\epsilon>0$ be given. By the absolute continuity of the Lebesgue integral, we can find $\delta_1>0$ such that
	\begin{equation}\label{eq66}
		\int_C\hat{d}(s)ds\leq\frac{\epsilon}{2}\ \mbox{for all}\ C\subseteq T\ \mbox{measurable with}\ |C|_1\leq\delta_1.
	\end{equation}
	
	Also, using (\ref{eq64}), we can find $\delta\in(0,\epsilon/2)$ such that
	\begin{equation}\label{eq67}
		\xi,\xi'\in K,|\xi-\xi'|_1\leq\delta\Rightarrow|\lambda_0(\xi)(t)-\lambda_0(\xi')(t)|\leq\hat{d}(t)\chi_{C_1}(t)\ \mbox{for almost all}\ t\in T,
	\end{equation}
	with $C_1\subseteq T$ measurable, $|C_1|_1\leq\delta_1$. So, returning to (\ref{eq65}) and using (\ref{eq66}) and (\ref{eq67}), we see that
	\begin{eqnarray*}
		&&\xi,\xi'\in K,|\xi-\xi'|\leq\delta\Rightarrow|\mu_1(\xi)(t)-\mu_1(\xi')(t)|\leq\frac{\epsilon}{2}+\int^t_0\hat{d}(s)\chi_{C_1}(s)ds\leq\frac{\epsilon}{2}+\frac{\epsilon}{2}=\epsilon\\
		&&\hspace{1cm}\mbox{for all}\ t\in T.
	\end{eqnarray*}
	
	Therefore $\xi\mapsto\mu_1(\xi)$ is continuous from $H$ into $C(T,H)$. Again, with an application of the Yankov-von Neumann-Aumann selection theorem, we obtain $\gamma_1(\xi)\in L^2(T,H)$ such that
	\begin{eqnarray*}
		&&\gamma_1(\xi)(t)\in F(t,\mu_1(\xi)(t))\ \mbox{for almost all}\ t\in T,\\
		&&|\gamma_0(\xi)(t)-\gamma_1(\xi)(t)|=d(\gamma_0(\xi)(t),F(t,\mu_1(\xi)(t)))\ \mbox{for almost all}\ t\in T,\ \mbox{all}\ \xi\in K.
	\end{eqnarray*}
	
	As in the proof of Proposition \ref{prop7},  we produce inductively two sequences
	$$\{\mu_n(\xi)\}_{n\geq 0}\subseteq W_p(0,b)\ \mbox{and}\ \{\gamma_n(\xi)\}_{n\geq 0}\subseteq L^2(T,H),\ \xi\in K,$$
	which satisfy the following properties:
	\begin{itemize}
		\item[(a)] $\mu_n(\xi)=e(\lambda_{n-1}(\xi),\xi)$ with $\lambda_{n-1}(\xi)=\overset{N}{\underset{\mathrm{k=0}}\sum}\chi_{T_k(\xi)}\gamma_{n-1}(\xi_k)(t),\gamma_{-1}(\xi)=f$;
		\item[(b)] $\xi\mapsto\mu_n(\xi)$ is continuous from $K$ into $C(T,H)$;
		\item[(c)] $|\mu_n(\xi)(t)-\mu_{n-1}(\xi)(t)|\leq\frac{\vartheta}{2^{n+2}n!}\left(\int^t_0k(s)ds\right)^n$ for all $\xi\in K$;
		\item[(d)] $\gamma_n(\xi)(t)\in F(t,\mu_n(\xi)(t))$ for almost all $t\in T$ and
		$$|\gamma_{n-1}(\xi)(t)-\gamma_n(\xi)(t)|=d(\gamma_{n-1}(\xi)(t),F(t,\mu_n(\xi)(t)))\ \mbox{for almost all}\ t\in T.$$
	\end{itemize}
	
	So, as induction hypothesis, suppose that we have produced
	$$\{\mu_k(\xi)\}^n_{k=0}\subseteq W_p(0,b)\ \mbox{and}\ \{\gamma_k(\xi)\}^n_{k=0}\subseteq L^2(T,H),$$
	which satisfy properties $(a)\rightarrow(d)$ stated above. We set
	$$\mu_{n+1}(\xi)=e(\lambda_n(\xi),\xi)\ \mbox{with}\ \lambda_n(\xi)(t)=\overset{n}{\underset{\mathrm{k=0}}\sum}\chi_{T_k(\xi)}(t)\gamma_n(\xi_k)(t).$$
	
	As above (see in the first part of the proof the argument concerning the map $\xi\mapsto\mu_1(\xi)$), we can show that $\xi\mapsto\mu_{n+1}(\xi)$ is continuous from $K$ into $C(T,H)$. Also, by the monotonicity of $A(t,\cdot)$ (see hypothesis $H(A)_2(ii)$ and Lemma A.5, p. 157, of Brezis \cite{6}), we have
	\begin{eqnarray*}
		 |\mu_{n+1}(\xi)(t)-\mu_n(\xi)(t)|&\leq&\overset{n}{\underset{\mathrm{k=0}}\sum}\int^t_0\chi_{T_k(\xi)}(s)k(s)|\mu_n(\xi)(s)-\mu_{n-1}(\xi)(s)|ds\\
		&&\left(\mbox{see hypothesis}\ H(F)_2(ii)\ \mbox{and property}\right.\ (d)\\
		&&\left.\mbox{of the induction hypothesis}\right)\\
		 &\leq&\overset{n}{\underset{\mathrm{k=0}}\sum}\int^t_0\chi_{T_k(\xi)}(s)k(s)\frac{\vartheta}{2^{n+2}n!}(\int^s_0k(s)dr)^nds\\
		&&(\mbox{see property $(c)$ of the induction hypothesis})\\
		&=&\int^t_0\frac{\vartheta}{2^{n+2}(n+2)!}\frac{d}{ds}(\int^s_0k(r)dr)^{n+1}ds\\
		&=&\frac{\vartheta}{2^{n+2}(n+1)!}(\int^t_0k(s)ds)^{n+1}.
	\end{eqnarray*}
	
	Moreover, a standard measurable selection argument, produces a measurable map\\ $\gamma_{n+1}(\xi):T\rightarrow H,\ \xi\in K$, such that
	\begin{eqnarray*}
		&&\gamma_{n+1}(\xi)(t)\in F(t,\mu_{n+1}(\xi)(t))\ \mbox{for almost all}\ t\in T,\\
		&&|\gamma_n(\xi)(t)-\gamma_{n+1}(\xi)(t)|=d(\gamma_n(\xi)(t),F(t,\mu_{n+1}(\xi)(t)))\ \mbox{for almost all}\ t\in T.
	\end{eqnarray*}
	
	This completes the induction process.
	
	Note that
	$$||\gamma_{n+1}(\xi)-\mu_n(\xi)||_{C(T,H)}\leq\frac{\vartheta}{2^{n+3}}e^{||k||_1}\ \mbox{(see property)}\ (c).$$
	
	Therefore, we can say that
	\begin{equation}\label{eq68}
		\mu_n(\xi)\rightarrow\psi(\xi)\ \mbox{in}\ C(T,H)\ \mbox{as}\ n\rightarrow\infty,\ \mbox{uniformly in}\ \xi\in K.
	\end{equation}
	
	It follows that $\xi\mapsto\psi(\xi)$ is continuous from $K$ into $C(T,H)$.
	
	Note that $T_0(\xi_0)=T=[0,b]$ and so $\mu_0(\xi_0)=e(f,\xi_0)=v$ (see (\ref{eq61})). Hence $\psi(\xi_0)=v$. It remains to show that $\psi$ is a selection of the solution multifunction $\xi\mapsto S(\xi)$. By property $(d)$ and hypothesis $H(F)_2(ii)$, we have
	\begin{eqnarray}\label{eq69}
		&&|\gamma_n(\xi)(t)-\gamma_{n+1}(\xi)(t)|\leq k(t)|\mu_n(\xi)(t)-\mu_{n+1}(\xi)(t)|\ \mbox{for almost all}\ t\in T,\nonumber\\
		&\Rightarrow&\gamma_{n+1}(\xi)\rightarrow\hat{\gamma}(\xi)\ \mbox{in}\ L^2(T,H).
	\end{eqnarray}
	
	Let $\hat{\mu}(\xi)=e\left(\overset{N}{\underset{\mathrm{k=0}}\sum}\int^t_0\chi_{T_k(\xi)}(s)\hat{\gamma}(\xi_k)(s)ds,\xi\right)$. Because
	 $$\overset{N}{\underset{\mathrm{k=0}}\sum}\chi_{T_k(\xi)}\gamma_n(\xi_k)\rightarrow\overset{N}{\underset{\mathrm{k=0}}\sum}\chi_{T_k(\xi)}\hat{\gamma}(\xi_k)\ \mbox{in}\ L^2(T,H)\ (\mbox{see (\ref{eq69})}),$$
	we have
	\begin{eqnarray*}
		&&\mu_n(\xi)\rightarrow\hat{\mu}(\xi)\ \mbox{in}\ C(T,H),\\
		&\Rightarrow&\psi(\xi)=\hat{\mu}(\xi)\in S(\xi)\ \mbox{for all}\ \xi\in K\ (\mbox{see (\ref{eq68})}).
	\end{eqnarray*}
\end{proof}

\section{Optimal Control Problems}

In this section we deal with the sensitivity analysis of the optimal control problem (\ref{eq1}).

Let $Q(\xi,\lambda)\subseteq W_p(0,b)\times L^2(T,Y)$ be the admissible ``state-control" pairs. First we investigate the dependence of this set on the initial condition $\xi\in H$ and the parameter $\lambda\in E$. Recall that the control space $Y$ is a separable reflexive Banach space and the parameter space $E$ is a compact metric space. To have a useful result on the dependence of $Q(\xi,\lambda)$ on $(\xi,\lambda)\in H\times E$, we introduce the following conditions on the data of the evolution inclusion in problem (\ref{eq1}) (the dynamical constraint of the problem).

\smallskip
$H(A)_3:$ $A:T\times X\times E\rightarrow 2^{X^*}\backslash\{\emptyset\}$ is a multifunction such that
\begin{itemize}
	\item[(i)] for every $(x,\lambda)\in X\times E,\ t\mapsto A_{\lambda}(t,x)$ is graph measurable;
	\item[(ii)] for almost all $t\in T$, all $\lambda\in E$, $x\mapsto A_{\lambda}(t,x)$ is maximal monotone;
	\item[(iii)] for almost all $t\in T$, all $x\in X$, all $\lambda\in E$ and all $h^*\in A_{\lambda}(t,x)$, we have
	$$||h^*||_*\leq a_{\lambda}(t)+c_{\lambda}||x||^{p-1}$$
	with $\{a_{\lambda}\}_{\lambda\in E}\subseteq L^{p'}(T)$ bounded $\{c_{\lambda}\}_{\lambda\in E}\subseteq(0,+\infty)$ bounded and $2\leq p<\infty$;
	\item[(iv)] for almost all $t\in T$, all $x\in X$, all $\lambda\in E$ and all $h^*\in A_{\lambda}(t,x)$, we have
	$$\left\langle h^*,x\right\rangle\geq\hat{c}||x||^p-\hat{a}(t)$$
	with $\hat{c}>0,\hat{a}\in L^1(T)_+$;
	\item[(v)] if $\lambda_n\rightarrow\lambda$ in $E$, then $\frac{d}{dt}+a_{\lambda_n}\xrightarrow{PG}\frac{d}{dt}+a_{\lambda}$ as $n\rightarrow\infty$.
\end{itemize}

\smallskip
Hypotheses $H(A)_3(i)\rightarrow(iv)$ are the same as hypotheses $H(A)_2(i)\rightarrow(iv)$ for every map $A_{\lambda}$, $\lambda\in E$. The new condition is hypothesis $H(A)_3(v)$, which requires elaboration. In the examples that follow, we present characteristic situations where this hypothesis is satisfied.
\begin{ex}\label{ex11}
	$(a)$ First, we present a situation which will be used in Section 5.
	
	So, let $\Omega\subseteq\RR^N$ be a bounded domain with Lipschitz boundary $\partial\Omega$. Let $X=W^{1,p}_0(\Omega)$ $(2\leq p<\infty)$, $H=L^2(\Omega)$, $X^*=W^{-1,p'}(\Omega)$. Evidently, $(X,H,X^*)$ is an evolution triple (see Definition \ref{def1}), with compact embeddings. We consider a map $a(t,z,\xi)$ satisfying the following conditions:
	
	$H(a):$ $a:T\times\Omega\times\RR^N\rightarrow\RR^N$ is a map such that
	\begin{itemize}
		\item[(i)] $|a(t,z,0)|\leq c_0$ for almost all $(t,z)\in T\times\Omega$;
		\item[(ii)] for every $\xi\in \RR^N,(t,z)\mapsto a(t,z,\xi)$ is measurable;
		\item[(iii)] $|a(t,z,\xi_1)-a(t,z,\xi_2)|\leq\hat{c}(1+|\xi_1|+|\xi_2|)^{p-1-\alpha}|\xi_1-\xi_2|^{\alpha}$ for almost all $(t,z)\in T\times\Omega$, all $\xi_1,\xi_2\in\RR^N$, with $\hat{c}_1>0$, $\alpha\in\left(0,1\right]$;
		\item[(iv)] $(a(t,z,\xi_1)-a(t,z,\xi_2),\xi_1-\xi_2)_{\RR^N}\geq\hat{c}_2|\xi_1-\xi_2|^p$ for almost all $(t,z)\in T\times\Omega$, all $\xi_1,\xi_2\in\RR^N$, $\xi_1\neq\xi_2$, with $\hat{c}_2>0$.
	\end{itemize}
	
	We consider the operator $A:T\times X\rightarrow X^*$ defined by
	$$\left\langle A(t,x),h\right\rangle=\int_{\Omega}(a(t,z,Dx),Dh)_{\RR^N}dz\ \mbox{for all}\ (t,x,h)\in T\times X\times X.$$
	
	Using the nonlinear Green's identity (see Gasinski and Papageorgiou \cite[p. 210]{20}), we have
	$$A(t,x)=-{\rm div}\,(B(t,x)),$$
	with $B(t,x)(\cdot)=a(t,\cdot,Dx(\cdot))\in L^{p'}(\Omega,\RR^N)$ for all $(t,x)\in T\times X$.
	
	Now consider a sequence $\{a_n(t,z,\xi)\}_{n\geq 1}$ of such maps satisfying
	\begin{eqnarray*}
		&&|a_n(t,z,\xi)-a_n(s,z,\xi)|\leq\vartheta(t-s)(1+|\xi|^{p-1})\\
		&&\mbox{for almost all}\ z\in\Omega,\ \mbox{all}\ t,s\in T,\ \mbox{all}\ \xi\in\RR^N,\ \mbox{all}\ n\in\NN,
	\end{eqnarray*}
	with $\vartheta:\RR_+\rightarrow\RR_+$ being an increasing function which is continuous at $r=0$ and $\vartheta(0)=0$. We assume that for almost all $t\in T$, $a_n(t,\cdot,\cdot)\stackrel{G}{\rightarrow}a(t,\cdot,\cdot)$ in the sense of Defranceschi \cite{12}. By Svanstedt \cite{39} we have
	$$\frac{d}{dt}+a_n\xrightarrow{PG}\frac{d}{dt}+a.$$
	
	$(b)$ We can allow multivalued maps, provided that we drop the $t$-dependence. So, we consider multivalued maps $a(z,\xi)$ which satisfy the following conditions:
	
	$H(a)':$ $a:\Omega\times\RR^N\rightarrow 2^{\RR^N}\backslash\{\emptyset\}$ is a measurable map such that
	\begin{itemize}
		\item[(i)] $a(\cdot,\cdot)$ is measurable;
		\item[(ii)] for almost all $z\in\Omega,\ \xi\mapsto a(z,\xi)$ is maximal monotone;
		\item[(iii)] for almost all $z\in\Omega$, all $\xi\in\RR^N$ and all $y\in a(z,\xi)$, we have
		\begin{eqnarray*}
			&&|y|^{p'}\leq m_1(z)+\tilde{c}_1(y,\xi)\ \mbox{with}\ m_1\in L^1(\Omega),\ \tilde{c}_1(y,\xi)>0,\\
			&&|\xi|^p\leq m_2(z)+\tilde{c}_2(y,\xi)\ \mbox{with}\ m_2\in L^1(\Omega),\ \tilde{c}_2(y,\xi)>0\ (2\leq p<\infty).
		\end{eqnarray*}
	\end{itemize}
	
	We again consider the evolution triple
	$$X=W^{1,p}_0(\Omega),H=L^2(\Omega),X^*=W^{-1,p'}(\Omega)\ (2\leq p<\infty)$$
	and consider the multivalued map $A:X\rightarrow 2^{X^*}\backslash\{\emptyset\}$ defined by
	$$A(x)=\{-{\rm div}\, g:g\in S^{p'}_{a(\cdot,Dx(\cdot))}\}.$$
	
	We consider a sequence $\{a_n(z,\xi)\}_{n\geq 1}$ of such maps and assume that $a_n\stackrel{G}{\rightarrow}a$ in the sense of Defranceschi \cite{12}. Then by Denkowski, Migorski and Papageorgiou \cite{14} we have
	$$\frac{d}{dt}+a_n\xrightarrow{PG}\frac{d}{dt}+a.$$
	
	$(c)$ A third situation leading to hypothesis $H(A)_3(v)$ is the following one. We consider maps $A_{\lambda}(t,x)$ satisfying the following conditions:
	
	$H(A)'_3:$ $A:T\times X\times E\rightarrow X^*$ is a map such that
	\begin{itemize}
		\item[(i)] $||A_{\lambda}(t+\tau,x)-A_{\lambda}(t,x)||\leq O(\tau)(1+||x||^{p-1})$ for all $t,t+\tau\in T$, all $x\in X$, all $\lambda\in E$;
		\item[(ii)] for all $(t,\lambda)\in T\times E,\ x\mapsto A_{\lambda}(t,x)$ is hemicontinous;
		\item[(iii)] for all $t\in T$, all $x,u\in X$, all $\lambda\in E$, we have
		$$\left\langle A_{\lambda}(t,x)-A_{\lambda}(t,u),x-u\right\rangle\geq\tilde{c}||x-u||^p\ \mbox{with}\ \tilde{c}>0;$$
		\item[(iv)] if $\lambda_n\rightarrow\lambda$ in $E$, then for all $t\in T$, $A_{\lambda_n}(t,\cdot)\stackrel{G}{\rightarrow}A_{\lambda}(t,\cdot)$ (this means that for all $x^*\in X^*$, $A^{-1}_{\lambda_n}(t,x^*)\stackrel{w}{\rightarrow}A^{-1}_{\lambda}(t,x)$, see Denkowski, Migorski and Papageorgiou \cite[Definition 3.8.20, p. 478]{16}).
	\end{itemize}
	
	Under these conditions, by Kolpakov \cite{26}, we have
	$$\frac{d}{dt}+a_{\lambda_n}\xrightarrow{PG}\frac{d}{dt}+a_{\lambda}.$$
	\end{ex}

	Next, we introduce the conditions on the multifunctions $F$ and $G$ involved in the dynamics of (\ref{eq1}).
	
\smallskip
	$H(F)_3:$ $F:T\times H\times E\rightarrow P_{f_c}(H)$ is a multifunction such that
	\begin{itemize}
		\item[(i)] for all $(x,\lambda)\in H\times E$, $t\mapsto F(t,x,\lambda)$ is graph measurable;
		\item[(ii)] for almost all $t\in T$, all $x,y\in H$ and all $\lambda\in E$, we have
		$$h(F(t,x,\lambda),F(t,y,\lambda))\leq k(t)|x-y|\ \mbox{with}\ k\in L^1(T)_+;$$
		\item[(iii)] for almost all $t\in T$, all $x\in H$ and all $\lambda\in E$, we have
		$$|F(t,x,\lambda)|\leq a_{\lambda}(t)+c_{\lambda}|x|$$
		with $\{a_{\lambda}\}_{\lambda\in E}\subseteq L^2(T)$ and $\{c_{\lambda}\}_{\lambda\in E}\subseteq(0,+\infty)$ bounded;
		\item[(iv)] for almost all $t\in T$, all $x\in H$ and all $\lambda,\lambda'\in E$, we have
		$$h(F(t,x,\lambda),F(t,x,\lambda'))\leq\beta(d(\lambda,\lambda'))w(t,|x|)$$
		with $\beta(r)\rightarrow 0^+$ as $r\rightarrow 0^+$ and $w(t,\cdot)$ bounded on bounded sets.
	\end{itemize}
	
\smallskip
	$H(G):$ $G:T\times Y\times E\rightarrow P_{f_c}(H)$ is a multifunction such that
	\begin{itemize}
		\item[(i)] for all $(u,\lambda)\in Y\times E,\ t\mapsto G(t,u,\lambda)$ is graph measurable;
		\item[(ii)] for almost all $t\in T$, all $\lambda\in E$, $u\mapsto G(t,u,\lambda)$ is concave (that is, ${\rm Gr}\,G(t,\cdot,\lambda)\subseteq Y\times H$ is concave, see Hu and Papageorgiou \cite{23}, Definition 1.1 and Remark 1.2, p. 585) and $(u,\lambda)\mapsto G(t,u,\lambda)$ is $h$-continuous;
		\item[(iii)] for almost all $t\in T$, all $u\in U(t,\lambda)$, all $\lambda\in E$
		$$|G(t,u,\lambda)|\leq\hat{a}_{\lambda}(t),$$
		with $\{\hat{a}_{\lambda}\}_{\lambda\in E}\subseteq L^2(T)$ bounded.
	\end{itemize}

\begin{remark}
	A typical situation resulting to a concave multifunction $u\mapsto G(t,u,\lambda)$, is when
	$$G(t,u,\lambda)=B_{\lambda}(t)u+C(t,\lambda)\ \mbox{for all}\ (t,u,\lambda)\in T\times Y\times E,$$
	with $B_{\lambda}(t)\in\mathcal{L}(Y,H)$ and $C(t,\lambda)\in P_{f_c}(H)$ for all $(t,\lambda)\in T\times E$.

Another situation, leading to the concavity of $G(t,\cdot,\lambda)$, is when $H$ is an ordered Hilbert space and $g_{\lambda},\tilde{g}_{\lambda}:T\times Y\rightarrow H$ are two Carath\'eodory maps such that for almost all $t\in T$
$$g_{\lambda}(t,\cdot)\ \mbox{is order convex and}\ \tilde{g}_{\lambda}(t,\cdot)\ \mbox{is order concave}.$$

We set $G(t,u,\lambda)=\{h\in H:g_{\lambda}(t,u)\leq h\leq \tilde{g}_{\lambda}(t,u)\}$. Then $G(t,\cdot,\lambda)$ is concave.
\end{remark}

Finally we impose conditions on the control constraint $U(t,\lambda)$.

\smallskip
$H(U):$ $U:T\times E\rightarrow P_{f_c}(Y)$ is a multifunction such that
\begin{itemize}
	\item[(i)] for all $\lambda\in E,t\mapsto  U(t,\lambda)$ is graph measurable;
	\item[(ii)] for almost all $t\in T,\ \lambda\mapsto U(t,\lambda)$ is $h$-continuous;
	\item[(iii)] $|U(t,\lambda)|\leq\tilde{a}_{\lambda}(t)$ for almost all $t\in T$, all $\lambda\in E$, with $\{\tilde{a}_{\lambda}\}_{\lambda\in E}\subseteq L^2(T)$ bounded.
\end{itemize}
\begin{prop}\label{prop12}
	If hypotheses $H(A)_3,H(F)_3,H(G),H(U)$ hold and $(\xi_n,\lambda_n)\rightarrow(\xi,\lambda)$ in $H\times E$, then
	\begin{eqnarray*}
		&&K_{seq}(s\times w)-\limsup\limits_{n\rightarrow\infty}Q(\xi_n,\lambda_n)\subseteq Q(\xi,\lambda)\ \mbox{in}\ L^p(T,H)\times L^2(T,Y),\\
		&&K(s\times s)-\liminf\limits_{n\rightarrow\infty}Q(\xi_n,\lambda_n)\supseteq Q(\xi,\lambda)\ \mbox{in}\ C(T,H)\times L^2(T,Y).
	\end{eqnarray*}
\end{prop}
\begin{proof}
	Let $(x,u)\in K_{seq}(s\times w)-\limsup\limits_{n\rightarrow\infty}Q(\xi_n,\lambda_n)$. By definition (see Section 2), we can find a subsequence $\{m\}$ of $\{n\}$ and $(x_m,u_m)\in Q(\xi_m,\lambda_m)$, $m\in\NN$ such that
	\begin{equation}\label{eq70}
		x_m\rightarrow x\ \mbox{in}\ L^p(T,H)\ \mbox{and}\ u_m\stackrel{w}{\rightarrow}u\ \mbox{in}\ L^2(T,Y)\ \mbox{as}\ m\rightarrow \infty.
	\end{equation}
	
	For every $m\in\NN$, we have
	\begin{equation}\label{eq71}
		-x'_m(t)\in A_{\lambda_m}(t,x_m(t))+f_m(t)+g_m(t)\ \mbox{for almost all}\ t\in T,\ x_m(0)=\xi_m,
	\end{equation}
	with $f_m,g_m\in L^2(T,H)$ such that
	\begin{equation}\label{eq72}
		f_m(t)\in F(t,x_m(t),\lambda_m)\ \mbox{and}\ g_m(t)\in G(t,u_m(t),\lambda_m)\ \mbox{for almost all}\ t\in T.
	\end{equation}
	
	We deduce by hypotheses $H(F)_3(iii),H(G)(iii)$ and Theorem \ref{th6} and its proof that
	$$\{x_m\}_{m\in\NN}\subseteq W_p(0,b)\ \mbox{is bounded and}\ \{x_m\}_{m\in\NN}\subseteq C(T,H)\ \mbox{is relatively compact.}$$
	
	So, from (\ref{eq70}) we obtain
	\begin{equation}\label{eq73}
		x_m\stackrel{w}{\rightarrow}x\ \mbox{in}\ W_p(0,b)\ \mbox{and}\ x_m\rightarrow x\ \mbox{in}\ C(T,H)\ \mbox{as}\ m\rightarrow\infty.
	\end{equation}
	
	By (\ref{eq72}) and hypotheses $H(F)_3(iii),H(G)(iii)$ it is clear that
	$$\{f_m\}_{m\in\NN},\{g_m\}_{m\in\NN}\subseteq L^2(T,H)\ \mbox{are bounded}.$$
	
	Hence, we may assume (at least for a subsequence), that
	\begin{equation}\label{eq74}
		f_m\stackrel{w}{\rightarrow}f\ \mbox{and}\ g_m\stackrel{w}{\rightarrow}g\ \mbox{in}\ L^2(T,H)\ \mbox{as}\ m\rightarrow\infty.
	\end{equation}
	
	Proposition 6.6.33 on p. 521 of Papageorgiou and Kyritsi \cite{34}, implies that
	\begin{equation}\label{eq75}
		f(t)\in\overline{\rm conv}\, w-\limsup\limits_{m\rightarrow\infty}F(t,x_m(t),\lambda_m)\ \mbox{for all}\ t\in T\backslash N,\ |N|_1=0.
	\end{equation}
	
	Fix $t\in T\backslash N$ and let $y\in w-\limsup\limits_{m\rightarrow\infty}F(t,x_m(t),\lambda_m)$. By definition, we know that there exists a subsequence $\{k\}$ of $\{m\}$ and $y_k\in F(t,x_k(t),\lambda_k)$ for all $k\in\NN$ such that $y_k\stackrel{w}{\rightarrow}y$ in $H$ as $k\rightarrow\infty$. The function $v\mapsto d(v,F(t,x(t),\lambda))$ is continuous and convex, hence weakly lower semicontinuous. Therefore
	\begin{equation}\label{eq76}
		d(y,F(t,x(t),\lambda))\leq\liminf\limits_{k\rightarrow\infty}d(y_k,F(t,x(t),\lambda)).
	\end{equation}
	
	On the other hand, we have
	\begin{equation}\label{eq77}
		d(y_k,F(t,x(t),\lambda))\leq h(F(t,x_k(t),\lambda_k),F(t,x(t),\lambda)).
	\end{equation}
	
	Using hypotheses $H(F)_3(ii),(iv)$, we have
	\begin{eqnarray*}
		&&h(F(t,x_k(t),\lambda_k),F(t,x(t),\lambda))\\
		&\leq&h(F(t,x_k(t),\lambda_k),F(t,x(t),\lambda_k))+h(F(t,x(t),\lambda_k),F(t,x(t),\lambda))\\
		&\leq&k(t)|x_k(t)-x(t)|+\beta(d(\lambda_k,\lambda))w(t,|x(t)|),\\
		\Rightarrow&&h(F(t,x_k(t),\lambda_k),F(t,x(t),\lambda))\rightarrow 0\ \mbox{as}\ k\rightarrow\infty\ (\mbox{see (\ref{eq73})}).
	\end{eqnarray*}
	
	Then we obtain from (\ref{eq76}) and (\ref{eq77}) that
	\begin{eqnarray*}
		&&d(y,F(t,x(t),\lambda))=0,\\
		&\Rightarrow &y\in F(t,x(t),\lambda),\\
		&\Rightarrow&w-\limsup\limits_{m\rightarrow\infty}F(t,x_m(t),\lambda_m)\subseteq F(t,x(t),\lambda)\ \mbox{for all}\ t\in T\backslash N,\ |N|_1=0,\\
		&\Rightarrow&f(t)\in F(t,x(t),\lambda)\ \mbox{for all}\ t\in T\backslash N,\ |N|_1=0\\
		&&\mbox{(see (\ref{eq75}) and recall that $F$ is convex-valued).}
	\end{eqnarray*}
	
	Next, for each $m\in\NN$, we have
	$$g_m\in S^2_{G(\cdot,u_m(\cdot),\lambda_m)}.$$
	
	Let $h\in L^2(T,H)$ and let $(\cdot,\cdot)_{L^2(T,H)}$ denote the inner product of $L^2(T,H)$ (recall that $L^2(T,H)^*=L^2(T,H)$). Then
	\begin{eqnarray}\label{eq78}
		 &&(h,g_m)_{L^2(T,H)}\leq\sigma(h,S^2_{G(\cdot,u_m(\cdot),\lambda_m)})=\int^b_0\sigma(h(t),G(t,u_m(t),\lambda_m))dt\\
		&&(\mbox{see Papageorgiou and Kyritsi \cite[Theorem 6.4.16, p. 492]{34}}).\nonumber
	\end{eqnarray}
	
	The concavity of $G(t,\cdot,\lambda)$ (see hypothesis $H(G)(ii)$), implies that the function $u\mapsto\sigma(h(t),G(t,u,\lambda))$ is concave. Since $E$ is a complete metric space,  it can be isometrically embedded, by the Arens-Eells theorem (see Gasinski and Papageorgiou \cite[Theorem 4.143, p. 655]{21}), as a closed subset of a separable Banach space (recall that $E$ is compact). So, by Balder \cite{3}, we have
	\begin{eqnarray*}
		 &&\limsup\limits_{m\rightarrow\infty}\int^b_0\sigma(h(t),F(t,u_m(t),\lambda_m))dt\leq\int^b_0\sigma(h(t),F(t,u(t),\lambda))dt\\
		&&(\mbox{see hypothesis}\ H(G)(ii)),\\
		 &\Rightarrow&\limsup\limits_{m\rightarrow\infty}\sigma(h,S^2_{G(\cdot,u_m(\cdot),\lambda_m)})\leq\sigma(h,S^2_{G(\cdot,u(\cdot),\lambda)}),\\
		&\Rightarrow&(h,g)_{L^2(T,H)}\leq\sigma(h,S^2_{G(\cdot,u(\cdot),\lambda)})\ (\mbox{see (\ref{eq74}), (\ref{eq78})}).
	\end{eqnarray*}
	
	Since $h\in L^2(T,H)$ is arbitrary, it follows that
	\begin{eqnarray*}
		&&g\in S^2_{G(\cdot,u(\cdot),\lambda)},\\
		&\Rightarrow&g(t)\in G(t,u(t),\lambda)\ \mbox{for almost all}\ t\in T.
	\end{eqnarray*}
	
	Let $y_m\in W_p(0,b)$ be the unique solution of the Cauchy problem
	\begin{equation}\label{eq79}
		-y'_m(t)\in A_{\lambda_m}(t,y_m(t))+f(t)+g(t)\ \mbox{for almost all}\ t\in T,\ y_m(0)=\xi.
	\end{equation}
	
	Hypothesis $H(A)_3(v)$ implies that
	\begin{equation}\label{eq80}
		y_m\stackrel{w}{\rightarrow}y\ \mbox{in}\ W_p(0,b)
	\end{equation}
	with $y\in W_p(0,b)$ being the unique solution of the Cauchy problem
	\begin{equation}\label{eq81}
		-y'(t)\in A_{\lambda}(t,y(t))+f(t)+g(t)\ \mbox{for almost all}\ t\in T,\ y(0)=\xi
	\end{equation}
	(see Section 2). From (\ref{eq71}) and (\ref{eq79}) and the monotonicity of $A_{\lambda_m}(t,\cdot)$ (see hypothesis $H(A)_3(ii)$), we have
	\begin{eqnarray*}
		&&\left\langle x'_m(t)-y'_m(t),x_m(t)-y_m(t)\right\rangle\leq(f(t)+g(t)-f_m(t)-g_m(t),x_m(t)-y_m(t))\\
		&&\mbox{for almost all}\ t\in T,\\
		 &\Rightarrow&\frac{1}{2}|x_m(t)-y_m(t)|^2\leq\frac{1}{2}|\xi_m-\xi|^2+\int^t_0(f(s)+g(s)-f_m(s)-g_m(s),x_m(s)-y_m(s))ds\\
		&&\mbox{for all}\ t\in T\ (\mbox{see Proposition \ref{prop2}})\\
		&\Rightarrow&||x_m-y_m||_{C(T,H)}\rightarrow 0\ \mbox{as}\ m\rightarrow\infty,\\
		&\Rightarrow&x=y\ (\mbox{see (\ref{eq73}), (\ref{eq80})}).
	\end{eqnarray*}
	
	Recalling that
	$$f(t)\in F(t,x(t),\lambda)\ \mbox{and}\ g(t)\in G(t,u(t),\lambda)\ \mbox{for almost all}\ t\in T,$$
it follows	from (\ref{eq81}) that
	\begin{eqnarray*}
		&&(x,u)\in Q(\xi,\lambda),\\
		&\Rightarrow&K_{seq}(s\times w)-\limsup\limits_{n\rightarrow\infty}Q(\xi_n,\lambda_n)\subseteq Q(\xi,\lambda)\ \mbox{in}\ L^p(T,H)\times L^2(T,Y).
	\end{eqnarray*}
	
	Next, we will prove the second convergence of the proposition.
	
	So, let $(x,u)\in Q(\xi,\lambda)$. By definition we have
	$$-x'(t)\in A_{\lambda}(t,x(t))+F(t,x(t),\lambda)+g(t)\ \mbox{for almost all}\ t\in T,\ x(0)=\xi,$$
	with $g\in L^2(T,H)$ satisfying
	$$g(t)\in G(t,u(t),\lambda)\ \mbox{for almost all}\ t\in T.$$
	
	For every $v\in L^2(T,Y)$, we have
	\begin{eqnarray*}
		&&d(v,S^2_{U(\cdot,\lambda_n)})=\int^b_0d(v(t),U(t,\lambda_n))dt,\\
		&&(\mbox{see Papageorgiou and Kyritsi \cite[Theorem 6.4.16, p. 492]{34}}).
	\end{eqnarray*}
	
	Hypothesis $H(U)(ii)$ and the dominated convergence theorem imply that
	\begin{eqnarray*}
		&&\int^b_0d(v(t),U(t,\lambda_n))dt\rightarrow\int^b_0d(v(t),U(t,\lambda))dt,\\
		&\Rightarrow&d(v,S^2_{U(\cdot,\lambda_n)})\rightarrow d(v,S^2_{U(\cdot,\lambda)}).
	\end{eqnarray*}
	
	Hence Proposition 6.6.22 on p. 518 of Papageorgiou and Kyritsi \cite{34} implies that we can find $u_n\in S^2_{U(\cdot,\lambda_n)}$ ($n\in\NN$) such that
	$$u_n\rightarrow u\ \mbox{in}\ L^2(T,Y)\ \mbox{as}\ n\rightarrow\infty\,.$$
	
	Then hypothesis $H(G)(ii)$ guarantees that we can find
	\begin{equation}\label{eq82}
		g_n\in L^2(T,H),g_n(t)\in G(t,u_n(t),\lambda_n)\ \mbox{for almost all}\ t\in T,\ \mbox{all}\ n\in\NN,
	\end{equation}
	such that
	$$g_n\rightarrow g\ \mbox{in}\ L^2(T,H)\ \mbox{as}\ n\rightarrow\infty.$$
	
	Given $\xi'\in H$, let $S(\xi')\subseteq W_p(0,b)$ be the set of solutions of the Cauchy problem
	$$-y'(t)\in A_{\lambda_n}(t,y(t))+F(t,y(t),\lambda)+g(t)\ \mbox{for almost all}\ t\in T,\ y(0)=\xi'.$$
	
	Let $K=\{\xi_n,\xi\}_{n\geq 1}\subseteq H$. This is a compact set in $H$. Invoking Proposition \ref{prop10} (with $\xi_0=\xi$), we produce a continuous map $\psi:K\rightarrow C(T,H)$ such that
	\begin{equation}\label{eq83}
		\hat{y}=\psi(\hat{\xi})\in S(\hat{\xi})\ \mbox{for all}\ \hat{\xi}\in H,\ \psi(\xi)=x.
	\end{equation}
	
	Let $y_n=\psi(\xi_n)$ ($n\in\NN$) and use Proposition \ref{prop8} to find $x_n\in W_p(0,b)$ solution of the Cauchy problem
	$$-x'_n(t)\in A_{\lambda_n}(t,x_n(t))+F(t,x_n(t),\lambda_n)+g_n(t)\ \mbox{for almost all}\ t\in T,\ x_n(0)=\epsilon_n,$$
	for which we have
	\begin{equation}\label{eq84}
		|x_n(t)-y_n(t)|\leq b\epsilon e^{\tau(t)}+\int^t_0\eta_n(s)e^{\tau(t)-\tau(s)}ds\ \mbox{for all}\ t\in T,
	\end{equation}
	with $\epsilon>0$, $\tau(t)=\int^t_0k(s)ds,\ \eta_n\in L^1(T),\ \eta_n\rightarrow 0$ in $L^1(T)$.
	
	So, we obtain
	$$\limsup\limits_{n\rightarrow\infty}||x_n-y_n||_{C(T,H)}\leq b\epsilon e^{\tau(b)}\ \mbox{(see (\ref{eq84})).}$$
	
	Since $\epsilon>0$ is arbitrary, it follows that
	\begin{equation}\label{eq85}
		||x_n-y_n||_{C(T,H)}\rightarrow 0\ \mbox{as}\ n\rightarrow\infty .
	\end{equation}
	
	Finally, we have
	\begin{eqnarray*}
		||x_n-x||_{C(T,H)}&\leq&||x_n-y_n||_{C(T,H)}+||y_n-x||_{C(T,H)}\\
		&=&||x_n-y_n||_{C(T,H)}+||\psi(\xi_n)-\psi(\xi)||_{C(T,H)}\ (\mbox{see (\ref{eq83})})\\
		&\Rightarrow&||x_n-x||_{C(T,H)}\rightarrow 0\ (\mbox{see (\ref{eq85}) and recall that}\ \psi(\cdot)\ \mbox{is continuous}).
	\end{eqnarray*}
	
	Since $(x_n,u_n)\in Q(\xi_n,\lambda_n)$ ($n\in\NN$) and $u_n\rightarrow u$ in $L^2(T,Y)$, we conclude that
	$$Q(\xi,\lambda)\subseteq K(s\times s)-\liminf\limits_{n\rightarrow\infty}Q(\xi_n,\lambda_n)\ \mbox{in}\ C(T,H)\times L^2(T,Y).$$
\end{proof}

An immediate consequence of the above proposition is the following corollary concerning the multifunction $(\xi,\lambda)\mapsto Q(\xi,\lambda)$ of admissible state-control pairs.
\begin{corollary}\label{cor13}
	If hypotheses $H(A)_3,H(F)_3,H(G),H(U)$ hold, then the multifunction $Q:H\times E\rightarrow 2^{C(T,H)\times L^2(T,Y)}\backslash\{\emptyset\}$ is lsc and sequentially closed in $C(T,H)\times L^2(T,Y)_w$ (that is, ${\rm Gr}\,Q\subseteq H\times E\times C(T,H)\times L^2(T,Y)_w$ is sequentially closed).
\end{corollary}

Now we bring the cost functional into the picture. The hypotheses on the integrands $L(t,x,\lambda)$ and $H(t,u,\lambda)$ are the following.

\smallskip
$H(L):$ $L:T\times H\times E\rightarrow\RR$ is an integrand such that
\begin{itemize}
	\item[(i)] for every $(x,\lambda)\in H\times E$, $t\mapsto L(t,x,\lambda)$ is measurable;
	\item[(ii)] if $\lambda_n\rightarrow \lambda$ in $E$, then for all $x\in H$ we have $L(\cdot,x,\lambda_n)\stackrel{w}{\rightarrow}L(\cdot,x,\lambda)$ in $L^1(T)$;
	\item[(iii)] for almost all $t\in T$, all $x,y\in H$ and all $\lambda\in E$, we have
	$$|L(t,x,\lambda)-L(t,y,\lambda)|=(1+|x|\vee|y|)\rho(t,|x-y|),$$
	where $|x|\vee|y|=\max\{|x|,|y|\}$ and $\rho(t,r)$ is a Carath\'eodory function on $T\times\RR_+$ with values in $(0,+\infty)$ such that
	$$\rho(t,0)=0\ \mbox{for almost all}\ t\in T\ \mbox{and}\ \sup[\rho(t,r):0\leq r\leq\vartheta]\leq\beta_{\vartheta}(t)\ \mbox{for almost all}\ t\in T,$$
	with $\beta_{\vartheta}\in L^1(T)_+,\vartheta>0$.
\end{itemize}

\smallskip
$H(H):$ $H:T\times Y\times E\rightarrow\RR$ is an integrand such that
\begin{itemize}
	\item[(i)] for all $(u,\lambda)\in Y\times E,\ t\mapsto H(t,u,\lambda)$ is measurable;
	\item[(ii)] for almost all $t\in T$, all $v\in E$, $u\mapsto H(t,u,\lambda)$ is convex and for almost all $t\in T$, all $u\in Y$, $\lambda\mapsto H(t,u,\lambda)$ is continuous;
	\item[(iii)] for almost all $t\in T$ and all $(u,\lambda)\in Y\times E$, we have
	$$H(t,u,\lambda)\leq a(t)(1+||u||^2_Y),$$
	with $a\in L^{\infty}(T)$.
\end{itemize}

\smallskip
$H(\hat{\psi}):$ $\hat{\psi}:H\times E\rightarrow\RR$ is a continuous function.

\smallskip
Using the direct method of the calculus of variations, we can produce optimal admissible state-control pairs for problem (\ref{eq1}).

\begin{prop}\label{prop14}
	If hypotheses $H(A)_3,H(F)_3,H(G),H(U),H(L),H(H)$ and $H(\hat{\psi})$ hold, then for every $(\xi,\lambda)\in H\times E$ we can find $(x^*,u^*)\in Q(\xi,\lambda)$ such that
	$$J(x^*,u^*,\xi,\lambda)=m(\xi,\lambda).$$
\end{prop}

\begin{proof}
	Let $\{(x_n,u_n)\}_{n\geq 1}\subseteq Q(\xi,\lambda)$ be a minimizing sequence for problem (\ref{eq1}). So, we have
	$$J(x_n,u_n,\xi,\lambda)\downarrow m(\xi,\lambda)\ \mbox{as}\ n\rightarrow\infty.$$
	
	Theorem \ref{th6} and hypothesis $H(U)$ imply that
	$$\{(x_n,u_n)\}_{n\geq 1}\subseteq W_p(0,b)\times L^2(T,Y)\ (\mbox{respectively,}\ \subseteq C(T,H)\times L^2(T,Y))$$
	is relatively $w\times w$-compact (respectively, $s\times w$-compact). So, by the Eberlein-Smulian theorem and by passing to a suitable subsequence if necessary, we can say that
	\begin{equation}\label{eq86}
		x_n\stackrel{w}{\rightarrow}x^*\ \mbox{in}\ W_p(0,b),\ x_n\rightarrow x^*\ \mbox{in}\ C(T,H),\ u_n\stackrel{w}{\rightarrow}u^*\ \mbox{in}\ L^2(T,Y).
	\end{equation}
	
	Then (\ref{eq86}) and Proposition \ref{prop12} imply that
	\begin{equation}\label{eq87}
		(x^*,u^*)\in Q(\xi,\lambda).
	\end{equation}
	
	Also, (\ref{eq86}), hypothesis $H(L)(iii)$ and the dominated convergence theorem, imply that
	\begin{equation}\label{eq88}
		\int^b_0L(t,x_n(t),\lambda)dt\rightarrow\int^b_0L(t,x^*(t),\lambda)dt.
	\end{equation}
	In addition, as before (see the proof of Proposition \ref{prop12}), using Theorem 2.1 of Balder \cite{3}, obtain
	\begin{equation}\label{eq89}
		\int^b_0H(t,u^*(t),\lambda)dt\leq\liminf\limits_{n\rightarrow\infty}\int^b_0H(t,u_n(t),\lambda)dt.
	\end{equation}
	
	Finally, (\ref{eq86}) and hypothesis $H(\hat{\psi})$ imply that
	\begin{equation}\label{eq90}
		\hat{\psi}(\xi,x_n(b),\lambda)\rightarrow\hat{\psi}(\xi,x^*(b),\lambda).
	\end{equation}
	
	We deduce from (\ref{eq87}), (\ref{eq88}), (\ref{eq89}), (\ref{eq90}) that
	$$J(x^*,u^*,\xi,\lambda)=m(\xi,\lambda)\ \mbox{with}\ (x^*,u^*)\in Q(\xi,\lambda).$$
This concludes the proof.
\end{proof}

We are now ready for the main sensitivity results concerning problem (\ref{eq1}). The first one establishes the Hadamard well-posedness of the problem.
\begin{theorem}\label{th15}
	If hypotheses $H(A)_3,\ H(F)_3,\ H(G),\ H(L),\ H(H)$ and $H(\hat{\psi})$ hold, then the value function $m:H\times E\rightarrow\RR$ of problem (\ref{eq1}) is continuous.
\end{theorem}
\begin{proof}
	Let $(\xi_n,\lambda_n)\rightarrow(\xi,\lambda)$ in $H\times E$. Let $(x,u)\in Q(\xi,\lambda)$ such that
	$$J(x,u,\xi,\lambda)=m(\xi,\lambda)\ (\mbox{see Proposition \ref{prop14}}).$$
	
	Invoking Proposition \ref{prop12}, we can find $(x_n,u_n)\in Q(\xi_n,\lambda_n)$ for all $n\in\NN$ such that
	\begin{equation}\label{eq91}
		x_n\rightarrow x\ \mbox{in}\ C(T,H)\ \mbox{and}\ u_n\rightarrow u\ \mbox{in}\ L^2(T,Y).
	\end{equation}
	
	We claim that
	\begin{equation}\label{eq92}
		\left|\int^b_0L(t,x_n(t),\lambda_n)dt-\int^b_0L(t,x(t),\lambda)dt\right|\rightarrow 0\ \mbox{as}\ n\rightarrow\infty.
	\end{equation}
	
	To this end note that
	\begin{eqnarray}\label{eq93}
		&&\left|\int^b_0L(t,x_n(t),\lambda_n)dt-\int^b_0L(t,x(t),\lambda)dt\right|\nonumber\\
		 &&\leq\left|\int^b_0L(t,x_n(t),\lambda_n)dt-\int^b_0L(t,x(t),\lambda_n)dt\right|+\left|\int^b_0L(t,x(t),\lambda_n)dt-\int^b_0L(t,x(t),\lambda)dt\right|\\
		&&\hspace{1cm}\mbox{for all}\ n\in\NN\,.\nonumber
	\end{eqnarray}
	
	First, we estimate the first summand in the right-hand side of (\ref{eq93}). Using hypothesis $H(L)(iii)$, we have
	\begin{eqnarray*}
		&&\left|\int^b_0L(t,x_n(t),\lambda_n)-L(t,x(t),\lambda_n)dt\right|\\
		&\leq&\int^b_0(1+|x_n(t)|\vee|x(t)|)\rho(t,|x_n(t)-x(t)|)dt.
	\end{eqnarray*}
	
	Let $M=\sup\limits_{n\geq 1}||x_n||_{C(T,H)}<+\infty$ (see (\ref{eq91})). Then
	\begin{eqnarray}\label{eq94}
		 &&\left|\int^b_0L(t,x_n(t),\lambda_n)dt-\int^b_0L(t,x(t),\lambda_n)dt\right|\leq(1+M)\int^b_0\rho(t,|x_n(t)-x(t)|)dt\rightarrow 0\\
		&& \mbox{as}\ n\rightarrow\infty\ (\mbox{see (\ref{eq91}) and hypothesis}\ H(L)(iii)).\nonumber
	\end{eqnarray}
	
	Next, we estimate the second term on the right-hand side of (\ref{eq93}).
	
	Let $\vartheta>2||x||_{C(T,H)}$ and let $\beta_{\vartheta}\in L^1(T)_+$ as postulated by hypothesis $H(L)(iii)$. Given $\epsilon>0$, we can find $\delta>0$ such that
	\begin{center}
		``if $C\subseteq T$ is measurable with $|C|_1\leq\delta$,\\
		\begin{equation}\label{eq95}
			\mbox{then}\ \int_C\beta_{\vartheta}(t)dt\leq\frac{\epsilon}{2(1+\vartheta)}."
		\end{equation}
	\end{center}
	
	Here, we use the absolute continuity of the Lebesgue integral. Invoking the Scorza-Dragoni theorem (see Papageorgiou and Kyritsi \cite[Theorem 6.2.9, p. 471]{34}), we can find $T_1\subseteq T$ closed with $|T\backslash T_1|\leq\frac{\delta}{2}$ and $\rho|_{T_1\times \RR_+}$ is continuous. Since $\rho(t,0)=0$, we can find $\delta_1>0$ such that
	\begin{equation}\label{eq96}
		``\mbox{if}\ r\in[0,\delta_1],\ \mbox{then}\ |\rho(t,r)|\leq\frac{\epsilon}{2b(1+\vartheta)}\ \mbox{for all}\ t\in T_1."
	\end{equation}
	
	Recall that simple functions are dense in $L^p(T,H)$. Using this fact, the property that $L^p(T,H)$-convergence implies pointwise convergence for almost all $t\in T$ for at least a subsequence and invoking Egorov's theorem, we can find $T_2\subseteq T$ closed and $s:T\rightarrow H$ a simple function such that
	\begin{equation}\label{eq97}
		||s||_{\infty}\leq||x||_{C(T,H)},\ |T\backslash T_2|_1\leq\frac{\delta}{2}\ \mbox{and}\ |x(t)-s(t)|\leq\delta_1\ \mbox{for all}\ t\in T_2.
	\end{equation}
	
	We set $T_3=T_1\cap T_2$. This is a closed subset of $T$ with $|T\backslash T_3|_1\leq\delta$. We have
	\begin{eqnarray}\label{eq98}
		&&\left|\int^b_0L(t,x(t),\lambda_n)dt-\int^b_0L(t,s(t),\lambda_n)dt\right|\nonumber\\
		&\leq&(1+||x||_{C(T,H)})\int^b_0\rho(t,|x(t)-s(t)|)dt\ (\mbox{see hypothesis}\ H(L)(iii)\ \mbox{and (\ref{eq97})})\nonumber\\
		&\leq&(1+\vartheta)\left[\int_{T_3}\rho(t,|x(t)-s(t)|)dt+\int_{T\backslash T_3}\rho(t,|x(t)-s(t)|)dt\right]\nonumber\\
		&\leq&\frac{\epsilon}{2}+\frac{\epsilon}{2}=\epsilon\ (\mbox{see (\ref{eq95}), (\ref{eq96}), (\ref{eq97})}).
	\end{eqnarray}
	
	Similarly, we show that
	\begin{equation}\label{eq99}
		\left|\int^b_0L(t,s(t),\lambda)dt-\int^b_0L(t,x(t),\lambda)dt\right|\leq\epsilon\,.
	\end{equation}
	
	Let $s(t)=\overset{N}{\underset{\mathrm{k=1}}\sum}v_{k}\chi_{C_k}(t)$ with $v_k\in H,C_k\subseteq T$ measurable. Using hypothesis $H(L)(ii)$, we can find $n_0\in\NN$ such that
	\begin{eqnarray}\label{eq100}
		 &&\left|\int^b_0(L(t,s(t),\lambda_n)-L(t,s(t),\lambda))dt\right|\leq\overset{N}{\underset{\mathrm{k=1}}\sum}\left|\int_{C_k}(L(t,v_k,\lambda_n)-L(t,v_k,\lambda))dt\right|\leq\epsilon\\
		&&\mbox{for all}\ n\geq n_0.\nonumber
	\end{eqnarray}
	
	It follows from (\ref{eq98}), (\ref{eq99}), (\ref{eq100}) that
	$$\int^b_0L(t,x(t),\lambda_n)dt\rightarrow\int^b_0L(t,x(t),\lambda)\ \mbox{as}\ n\rightarrow\infty.$$
	
	This convergence and (\ref{eq94}) imply that (\ref{eq92}) (our claim) is true.
	
	Next, we consider the integral functional
	$$\Phi(u,\lambda)=\int^b_0H(t,u(t),\lambda)dt\ \mbox{for all}\ (u,\lambda)\in L^2(T,Y)\times E.$$
	
	For every $\lambda\in E,u\mapsto\Phi(u,\lambda)$ is convex (see hypothesis $H(H)(ii)$). Also, hypothesis $H(H)$ (iii) implies that in a neighborhood of every $u\in L^2(T,Y)$
	$$\{\Phi(\cdot,\lambda)\}_{\lambda\in E}\ \mbox{is equibounded above},$$
	hence
	$$\{\Phi(\cdot,\lambda)\}_{\lambda\in E}\ \mbox{is equi-locally Lipschitz}$$
	(see Papageorgiou and Kyritsi \cite[Theorem 1.2.3, p. 13]{34}). Therefore it follows that
	\begin{equation}\label{eq101}
		\Phi(u_n,\lambda_n)\rightarrow\Phi(u,\lambda)\ \mbox{as}\ n\rightarrow\infty\ (\mbox{see (\ref{eq91})}).
	\end{equation}
	
	Finally, (\ref{eq91}) and hypothesis $H(\hat{\psi})$ imply that
	\begin{equation}\label{eq102}
		\hat{\psi}(\xi_n,x_n(b),\lambda_n)\rightarrow\hat{\psi}(\xi,x(b),\lambda).
	\end{equation}
	
	By (\ref{eq92}), (\ref{eq101}), (\ref{eq102}), we have
	\begin{eqnarray}\label{eq103}
		&&\int^b_0L(t,x_n(t),\lambda_n)dt+\int^b_0H(t,u_n(t),\lambda_n)dt+\hat{\psi}(\xi_n,x_n(b),\lambda_n)\nonumber\\
		&&\hspace{5cm}\rightarrow J(x,u,\xi,\lambda)=m(\xi,\lambda),\nonumber\\
		&\Rightarrow&\limsup_{n\rightarrow\infty}m(\xi_n,\lambda_n)\leq m(\xi,\lambda).
	\end{eqnarray}
	
	From Proposition \ref{prop14} we know that for every $n\in\NN$, we can find $(x_n,u_n)\in Q(\xi_n,\lambda_n)$ such that
	\begin{equation}\label{eq104}
		J(x_n,u_n,\xi_n,\lambda_n)=m(\xi_n,\lambda_n).
	\end{equation}
	
	As in the proof of Theorem \ref{th6}, we can show that $\{x_n\}_{n\geq 1}\subseteq W_p(0,b)$ is bounded. In addition, hypothesis $H(U)$ implies that $\{u_n\}_{n\geq 1}\subseteq L^2(T,Y)$ is bounded. So, by passing to a suitable subsequence if necessary, we may assume that
	\begin{equation}\label{eq105}
		x_n\stackrel{w}{\rightarrow}x\ \mbox{in}\ W_p(0,b)\ \mbox{and}\ u_n\stackrel{w}{\rightarrow}u\ \mbox{in}\ L^2(T,Y)\ \mbox{as}\ n\rightarrow\infty\,.
	\end{equation}
	
	By (\ref{eq105}) and (\ref{eq4}), we also have
	\begin{equation}\label{eq106}
		x_n\rightarrow x\ \mbox{in}\ L^p(T,H)\ \mbox{as}\ n\rightarrow\infty\,.
	\end{equation}
	
	Then (\ref{eq105}), (\ref{eq106}) and Proposition \ref{prop12} imply that
	$$(x,u)\in Q(\xi,\lambda).$$
	
	Moreover, reasoning as in the proof of Theorem \ref{th6}, we show that
	\begin{eqnarray}\label{eq107}
		&&\{x_n\}_{n\geq 1}\subseteq C(T,H)\ \mbox{is relatively compact},\nonumber\\
		&\Rightarrow&x_n\rightarrow x\ \mbox{in}\ C(T,H)\ (\mbox{see (\ref{eq106})}).
	\end{eqnarray}
	
	By (\ref{eq107}) and the first part of the proof, we have
	$$\int^b_0L(t,x_n(t),\lambda_n)dt\rightarrow\int^b_0L(t,x(t),\lambda)dt.$$
	
	In addition, (\ref{eq105}) and hypotheses $H(H)(ii),\ H(\hat{\psi})$ imply
	\begin{eqnarray*}
		&&\int^b_0H(t,u(t),\lambda)dt\leq\liminf\limits_{n\rightarrow \infty}\int^b_0H(t,u_n(t),\lambda_n)dt\ (\mbox{see Balder \cite{3}})\\
		&&\hat{\psi}(\xi_n,x_n(b),\lambda_n)\rightarrow\hat{\psi}(\xi,x(b),\lambda).
	\end{eqnarray*}
	
	Therefore, from (\ref{eq104}) we see that
	\begin{eqnarray}\label{eq108}
		 &&\int^b_0L(t,x(t),\lambda)dt+\int^b_0H(t,u(t),\lambda)dt+\hat{\psi}(\xi,x(b),\lambda)\leq\liminf\limits_{n\rightarrow\infty}m(\xi_n,\lambda_n),\nonumber\\
		&\Rightarrow&m(\xi,\lambda)\leq\liminf\limits_{n\rightarrow\infty}m(\xi_n,\lambda_n).
	\end{eqnarray}
	
	We infer from (\ref{eq103}) and (\ref{eq108}) that
	\begin{eqnarray*}
		&&m(\xi_n,\lambda_n)\rightarrow m(\xi,\lambda),\\
		&\Rightarrow&m:H\times E\rightarrow\RR\ \mbox{is continuous.}
	\end{eqnarray*}
\end{proof}

For every $(\xi,\lambda)\in H\times E$, we introduce the set $\Sigma(\xi,\lambda)$ of optimal state-control pairs, that is,
$$\Sigma(\xi,\lambda)=\{(x,u)\in Q(\xi,\lambda):J(x,u,\xi,\lambda)=m(\xi,\lambda)\}.$$

By Proposition \ref{prop14}, we know that for every $(\xi,\lambda)\in H\times E$, $\Sigma(\xi,\lambda)\neq\emptyset$. For this multifunction we can prove the following useful continuity property.

\begin{theorem}\label{th16}
	If hypotheses $H(A)_3,\ H(F)_3,\ H(U),\ H(L),\ H(H)$ and $H(\hat{\psi})$ hold, then the multifunction $\Sigma:H\times E\rightarrow 2^{C(T,H)\times L^2(T,Y)}\backslash\{\emptyset\}$ is sequentially usc into $C(T,H)\times L^2(T,Y)_w$.
\end{theorem}
\begin{proof}
	Let $C\subseteq C(T,H)\times L^2(T,Y)_w$ be sequentially closed. We need to show that
	$$\Sigma^-(C)=\{(\xi,\lambda)\in H\times E:V(\xi,\lambda)\cap C\neq\emptyset\}$$
	is closed in $H\times E$ (see Section 2). To this end, let $\{(\xi_n,\lambda_n)\}_{n\geq 1}\subseteq\Sigma^-(C)$ and assume that
	$$(\xi_n,\lambda_n)\rightarrow(\xi,\lambda)\ \mbox{in}\ H\times E.$$
	
	Let $(x_n,u_n)\in \Sigma(\xi_n,\lambda_n)\cap C$, $n\in\NN$. We know from the proof of Theorem \ref{th15} that at least for a subsequence, we have
	\begin{equation}\label{eq109}
		x_n\stackrel{w}{\rightarrow}x\ \mbox{in}\ W_p(0,b),\ x_n\rightarrow x\ \mbox{in}\ C(T,H),\ u_n\stackrel{w}{\rightarrow}u\ \mbox{in}\ L^2(T,Y)\ \mbox{as}\ n\rightarrow\infty\,.
	\end{equation}
	
	By (\ref{eq109}) and Proposition \ref{prop12}, we have
	\begin{equation}\label{eq110}
		(x,u)\in Q(\xi,\lambda).
	\end{equation}
	
	Also, we know from the proof of Theorem \ref{th15} that
	\begin{eqnarray*}
		 &&J(x,u,\xi,\lambda)\leq\liminf\limits_{n\rightarrow\infty}J(x_n,u_n,\xi_n,\lambda_n)=\liminf\limits_{n\rightarrow\infty}m(\xi_n,\lambda)=m(\xi,\lambda)\\
		&&\mbox{(see Theorem \ref{th15})},\\
		&\Rightarrow&J(x,u,\xi,\lambda)=m(\xi,\lambda)\ (\mbox{see (\ref{eq110})}),\\
		&\Rightarrow&(x,u)\in\Sigma(\xi,\lambda).
	\end{eqnarray*}
	
	Moreover, from (\ref{eq109}) and since $C\subseteq C(T,H)\times L^2(T,Y)_w$ is sequentially closed, we deduce that $(x,u)\in\Sigma(\xi,\lambda)\cap C$. Therefore $\Sigma^-(C)\subseteq H\times E$ is closed and this proves the desired sequential upper semicontinuity of the multifunction $(\xi,\lambda)\mapsto \Sigma(\xi,\lambda)$.
\end{proof}

\section{Application to Distributed Parameter Systems}

In this section we present an application to a class of multivalued parabolic optimal control problems.

So, let $T=[0,b]$ and let $\Omega\subseteq\RR^N$ be a bounded domain with a Lipschitz boundary $\partial\Omega$. We examine the following nonlinear, multivalued parabolic optimal control problem:
\begin{eqnarray}\label{eq111}
	\left\{\begin{array}{l}
		J(x,u,\xi,\lambda)=\int^b_0\int_{\Omega}L_1(t,z,x(t,z))dzdt+\int^b_0\int_{\Omega}H_1(t,z,u(t,z))dzdt\rightarrow\\
			\inf=m(\xi,\lambda)\\
		\mbox{such that}\ -\frac{\partial x}{\partial t}\in-{\rm div}\,a_{\lambda}(z,Du)+F_1(t,z,x(t,z),\lambda)+g(t,z,\lambda)u(t,z)\\
		\mbox{on}\ T\times \Omega\\
		x|_{T\times\partial\Omega}=0,x(0,z)=\xi(z)\ \mbox{for almost all}\ z\in\Omega,\\
		||u(t,\cdot)||_{L^2(\Omega)}\leq r(t,\lambda)\ \mbox{for almost all}\ t\in T.
	\end{array}\right\}
\end{eqnarray}

Here, $a_{\lambda}:\Omega\times\RR^N\rightarrow 2^{\RR^N}$ ($\lambda\in E$) is a family of multifunctions as in Example \ref{ex11}(b). For the other data of problem (\ref{eq111}), we introduce the following conditions:

\smallskip
$H(F_1):$ $F_1:T\times\Omega\times\RR\times E\rightarrow P_{f_c}(\RR)$ is a multifunction such that
\begin{itemize}
	\item[(i)] for all $(x,\lambda)\in\RR\times E,(t,z)\mapsto F_1(t,z,x,\lambda)$ is measurable;
	\item[(ii)] for almost all $(t,z)\in T\times\Omega$, all $x,y\in\RR$, all $\lambda\in E$, we have
	$$h(F_1(t,z,x,\lambda),F_1(t,z,y,\lambda))\leq k_1(t,z)|x-y|,$$
	with $k_1\in L^1(T,L^{\infty}(\Omega))$;
	\item[(iii)] for almost all $(t,z)\in T\times \Omega$, all $x\in\RR$, all $\lambda\in E$, we have
	$$|F_1(t,z,x,\lambda)|\leq\hat{a}_1(t,z)+\hat{c}|x|,$$
	with $\hat{a}_1\in L^2(T\times\Omega),\hat{c}_1>0;$
	\item[(iv)] for almost all $(t,z)\in T\times\Omega$, all $x\in\RR$, all $\lambda,\lambda'\in E$, we have
	$$h(F_1(t,z,x,\lambda),F_1(t,z,x,\lambda'))\leq\beta(d(\lambda,\lambda'))w(z,|x|),$$
	with $\beta(r)\rightarrow 0$ as $r\rightarrow 0^+$ and $w\in L^{\infty}_{loc}(\Omega\times\RR_+)$.
\end{itemize}

\begin{remark}
	Consider the multifunction $F(t,z,x,\lambda)$ defined by
	$$F(t,z,x,\lambda)=[f(t,z,x,\lambda),\hat{f}(t,z,x,\lambda)]$$
	with $f,\hat{f}:T\times\Omega\times\RR\times E\rightarrow\RR$ two functions such that
	\begin{itemize}
		\item for all $(x,\lambda)\in\RR\times E$, $(t,z)\mapsto f(t,z,x,\lambda),\hat{f}(t,z,x,\lambda)$ are both measurable;
		\item for almost all $(t,z)\in T\times\Omega$, all $x,x'\in\RR$, all $\lambda,\lambda'\in E$, we have
		\begin{eqnarray*}
			&&|f(t,z,x,\lambda)-f(t,z,x',\lambda')|\leq k(t,z)[|x-x'|+d(\lambda,\lambda')]\\
			&&|\hat{f}(t,z,x,\lambda)-\hat{f}(t,z,x',\lambda')|\leq\hat{k}(t,z)[|x-x'|+d(\lambda,\lambda')],
		\end{eqnarray*}
		with $k,\hat{k}\in L^1(T,L^{\infty}(\Omega))$.
	\end{itemize}
\end{remark}

Then this multifunction satisfies hypotheses $H(F_1)$.

\smallskip
$H(g):$ $g:T\times\Omega\times E\rightarrow\RR$ is a Carath\'eodory function (that is, for all $\lambda\in E$, $(t,z)\rightarrow g(t,z,\lambda)$ is measurable and for almost all $(t,z)\in T\times\Omega$, $\lambda\rightarrow g(t,z,\lambda)$ is continuous) and for almost all $(t,z)\in T\times\Omega$ and all $\lambda\in E$, we have $|g(t,z,\lambda)|\leq M$ with $M>0$.

\smallskip
$H(r):$ $r:T\times E\rightarrow\RR_+$ is a Carath\'eodory function (that is, for all $\lambda\in E$, $t\mapsto r(t,\lambda)$ is measurable and for almost all $t\in T$, $\lambda\rightarrow r(t,\lambda)$ is continuous) and for almost all $t\in T$, all $\lambda\in E$, we have
$$0\leq r(t,\lambda)\leq a(t),$$
with $a\in L^2(T).$

Now, we introduce the conditions on the two integrands involved in the cost functional problem (\ref{eq111}).

\smallskip
$H(L_1):$ $L:T\times\Omega\times\RR\times E\rightarrow\RR$ is an integrand such that
\begin{itemize}
	\item[(i)] for all $(x,\lambda)\in\RR\times E,(t,z)\mapsto L_1(t,z,x,\lambda)$ is measurable;
	\item[(ii)] if $\lambda_n\rightarrow\lambda$ in $E$, then for all $x\in L^2(\Omega)$ we have $L_1(\cdot,\cdot,x(\cdot),\lambda_n)\stackrel{w}{\rightarrow}L_1(\cdot,\cdot,x(\cdot),\lambda)$ in $L^1(T\times\Omega)$;
	\item[(iii)] for almost all $(t,z)\in T\times\Omega$, all $x,y\in\RR$, all $\lambda\in E$
	$$|L_1(t,z,x,\lambda)-L_1(t,z,y,\lambda)|\leq c(1+|x|\vee|y|)\rho(t,z,|x-y|),$$
	with $\rho(t,z,r)$ Carath\'eodory, $\rho(t,z,0)=0$ for almost all $(t,z)\in T\times\Omega$ and for almost all $(t,z)$, all $r\in[0,\vartheta]$ we have
	$$0\leq\rho(t,z,r)\leq\beta_{\vartheta}(t,z)$$
	with $\beta_{\vartheta}\in L^1(T\times\Omega)$.
\end{itemize}

\smallskip
$H(H)_1:$ $H_1:T\times\Omega\times\RR\times E\rightarrow\RR$ is an integrand such that
\begin{itemize}
	\item[(i)] for all $(x,\lambda)\in\RR\times E$, $(t,z)\mapsto H_1(t,z,x,\lambda)$ is measurable;
	\item[(ii)] for almost all $(t,z)\in T\times\Omega$, $u\mapsto H_1(t,z,u,\lambda)$ is convex for all $\lambda\in E$, while $\lambda\mapsto H_1(t,z,u,\lambda)$ is continuous for all $u\in\RR$;
	\item[(iii)] for almost all $(t,z)\in T\times\Omega$, all $|u|\leq r_{\lambda}(t,z)$, all $\lambda\in E$, we have
	$$|H_1(t,z,u,\lambda)|\leq\hat{a}_{\lambda}(t,z),$$
	with $\{\hat{a}_{\lambda}\}_{\lambda\in E}\subseteq L^2(T\times\Omega)$ bounded.
\end{itemize}

\smallskip
We consider the following evolution triple:
$$X=W^{1,p}_{0}(\Omega),\ H=L^2(\Omega),\ X^*=W^{-1,p'}(\Omega).$$

Since $2\leq p<\infty$, the Sobolev embedding theorem implies that in this triple the embeddings are compact.

For every $\lambda\in E$, let $A_{\lambda}:X\rightarrow 2^{X^*}\backslash\{\emptyset\}$ be the multivalued map defined by
$$A_{\lambda}(x)=\{-{\rm div}\,g:\ g\in L^{p'}(\Omega,\RR^N),\ g(z)\in a_{\lambda}(z,Dx(z))\ \mbox{for almost all}\ z\in\Omega\}.$$

This map is maximal monotone and if $\lambda_n\rightarrow\lambda$ in $E$, then
$$\frac{d}{dt}+a_{\lambda_n}\xrightarrow{PG}\frac{d}{dt}+a_{\lambda}$$
(see Example \ref{ex11}(b)). So, hypotheses $H(A)_3$ hold. In fact, we can have $t$-dependence at the expense of assuming that $a_{\lambda}$ is single-valued. So, we assume that $a_{\lambda}(t,z,\xi)$ satisfies the conditions of Example \ref{ex11} (a). Then the map $A_{\lambda}:T\times X\rightarrow X^*$ is defined by
$$A_{\lambda}(t,x)(\cdot)=-{\rm div}\,a_{\lambda}(t,\cdot,Dx(\cdot)).$$

In fact, by the nonlinear Green's identity (see Gasinski and Papageorgiou \cite[p. 210]{20}), we have
$$\left\langle A_{\lambda}(t,x),h\right\rangle=\int_{\Omega}(a_{\lambda}(t,z,Dx),Dh)_{\RR^N}dz\ \mbox{for all}\ x,h\in W^{1,p}_{0}(\Omega).$$

As we have already mentioned in Example \ref{ex11}(a), we know from Svanstedt \cite{39} that if $\lambda_n\rightarrow\lambda$ in $E$, then
$$\frac{d}{dt}+a_{\lambda_n}\xrightarrow{PG}\frac{d}{dt}+a_{\lambda}$$
and so hypotheses $H(A)_3$ hold.

As a special case of interest, we consider the situation where the elliptic differential operator is a weighted $p$-Laplacian, that is,
$${\rm div}\,(a_{\lambda}(t,z)|Dx|^{p-2}Dx)\ \mbox{for all}\ x\in W^{1,p}_{0}(\Omega).$$

Here, for every $\lambda\in E,\ a_{\lambda}:T\times\Omega\rightarrow\RR$ is a measurable function such that
\begin{itemize}
	\item	$0<\hat{c}_1\leq a_{\lambda}(t,z)\leq\hat{c}_2$ for almost all $(t,z)\in T\times\Omega$, all $\lambda\in E$;
	\item if $\lambda_n\rightarrow\lambda$ in $E$, then for almost all $t\in T$, $$\frac{1}{a_{\lambda_n}(t,\cdot)^{p'-1}}\stackrel{w}{\rightarrow}\frac{1}{a_{\lambda}(t,\cdot)^{p'-1}} \quad\mbox{in}\ L^1(\Omega).$$
\end{itemize}

For this case we consider the following parametric (with parameter $\lambda\in E$) family of convex (in $\xi\in\RR^N$) integrands:
$$\varphi_{\lambda}(t,z,\xi)=\frac{a_{\lambda}(t,z)}{p}|\xi|^p.$$

Then the convex conjugate of $\varphi_{\lambda}(t,z,\cdot)$ is given by
$$\varphi^*_{\lambda}(t,z,\xi^*)=\frac{1}{p'a_{\lambda}(t,z)^{p'-1}}|\xi^*|^{p'}.$$

By hypothesis we have that
\begin{eqnarray}\label{eq112}
	&&\lambda_n\rightarrow \lambda\ \mbox{in}\ E\Rightarrow\varphi^*_{\lambda_n}(t,\cdot,\xi^*)\rightarrow\varphi^*_{\lambda}(t,\cdot,\xi^*)\ \mbox{in}\ L^1(\Omega)\\
	&&\mbox{for almost all}\ t\in T,\ \mbox{all}\ \xi^*\in\RR^N.\nonumber
\end{eqnarray}

We introduce the integral functional $\Phi_{\lambda}$ defined by
$$\Phi_{\lambda}(t,x)=\int_{\Omega}\varphi_{\lambda}(t,z,Dx)dz\ \mbox{for all}\ (t,x)\in T\times W^{1,p}_{0}(\Omega).$$

By Marcellini and Sbordone \cite{29}, we know that (\ref{eq112}) implies
$$\Phi_{\lambda}(t,x)=\Gamma_{seq}(w)-\Phi_{\lambda_n}(t,x)$$
with $\Gamma_{seq}(w)$ denoting the sequential $\Gamma$-convergence of $\Phi_{\lambda_n}(t,\cdot)$ on $W^{1,p}_{0}(\Omega)_w$ (see Buttazzo \cite{7}). Then it follows from Defranceschi \cite[Theorem 3.3]{12} that
$$a_{\lambda_n}(t,\cdot,\cdot)\stackrel{G}{\rightarrow}a_{\lambda}(t,\cdot,\cdot)\ \mbox{for almost all}\ t\in T$$
and so we conclude from Svanstedt \cite{39} that
$$\frac{d}{dt}+a_{\lambda_n}\xrightarrow{PG}\frac{d}{dt}+a_{\lambda}.$$

Also, let, $Y=H=L^2(\Omega)$ and
\begin{eqnarray*}
	 &&F(t,x,\lambda)=S^2_{F_1(t,\cdot,x(\cdot),\lambda)},\ G(t,u,\lambda)=\{g(t,\cdot,\lambda)u(\cdot):||u||_{L^2(\Omega)}\leq r(t,\lambda)\}\\
	&&U(t,\lambda)=\{u\in L^2(\Omega):||u||_{L^2(\Omega)}\leq r(t,\lambda)\}.
\end{eqnarray*}

Then hypotheses $H(F_1),\ H(g),\ H(r)$ imply that conditions $H(F)_3,\ H(G),\ H(U)$ hold. So, the dynamics of (\ref{eq111}) are described by an evolution inclusion similar to the one in problem (\ref{eq1}).

Finally let
\begin{eqnarray*}
	&&L(t,x,\lambda)=\int_{\Omega}L_1(t,z,x(z),\lambda)dz\ \mbox{for all}\ x\in L^2(\Omega),\\
	&&H(t,u,\lambda)=\int_{\Omega}H_1(t,z,u(z),\lambda)dz\ \mbox{dor all}\ u\in L^2(\Omega).
\end{eqnarray*}

Hypotheses $H(L_1),\ H(H_1)$ imply that conditions $H(L),\ H(H)$ respectively hold.

So, we can apply Theorems \ref{th15} and \ref{th16} and obtain the following result concerning the variational stability of problem (\ref{eq111}).
\begin{prop}\label{prop17}
	If the maps $a_{\lambda}$ are as above and hypotheses $H(F_1),\ H(g),\ H(r),\ H(L_1)$, $H(H_1)$ hold, then for every $(\xi,\lambda)\in L^2(\Omega)\times E$, problem (\ref{eq111}) admits optimal pairs (that is, $\Sigma(\xi,\lambda)\neq\emptyset$) and
\begin{eqnarray*}
	&&(\xi,\lambda)\mapsto m(\xi,\lambda)\ \mbox{is continuous on}\ L^2(\Omega)\times E;\\
	&&(\xi,\lambda)\mapsto\Sigma(\xi,\lambda)\ \mbox{is sequentially usc from}\ L^2(\Omega)\times E\ \mbox{into}\ C(T,L^2(\Omega))\times L^2(T\times\Omega)_w.
\end{eqnarray*}
\end{prop}

\medskip
{\bf Acknowledgements.} This research was supported in part by the SRA grants P1-0292-0101, J1-6721-0101 and J1-7025-0101.


\begin{thebibliography}{99}


\bibitem{1} J.-P. Aubin, A. Cellina, {\it Differential Inclusions}, Springer-Verlag, Berlin (1984).

\bibitem{2} J.-P. Aubin, H. Frankowska, {\it Set Valued Analysis}, Birkh\"auser, Boston (1990).

\bibitem{3} E. Balder,  Necessary and sufficient conditions for $L_{1}$-strong-weak lower semicontinuity of integral functionals, {\it Nonlinear Anal.} {\bf 11} (1987), 1399-1404.

\bibitem{4} V. Barbu, {\it Nonlinear Semigroups and Differential Equations in Banach Spaces}, Noordhoff International, Leyden, The Netherlands (1976).

\bibitem{5} A. Bressan, G. Colombo,  Extensions and selections of maps with decomposable values, {\it Studia Math.} {\bf 90} (1988), 69-85.

\bibitem{6} H. Brezis, {\it Op\'erateurs Maximaux Monotones et Semi-Groupes de Contractions dans les Espaces de Hilbert}, North Holland, Amsterdam (1973).

\bibitem{7} G. Buttazzo, {\it Semicontinuity, Relaxation and Integral Representation in the Calculus of Variations}, Pitman Research Notes in Mathematics Series, Vol. 207, Longman Scientific \& Technical, Harlow; copublished in the United States with John Wiley \& Sons, Inc., New York (1989).

\bibitem{8} G. Buttazzo, G. Dal Maso,  $\Gamma$-convergence and optimal control problems, {\it J. Optim. Theory. Appl.} {\bf 38} (1982), 385-407.

\bibitem{9} A. Cellina, A. Ornelas,  Representation of the attainable set for Lipschitzian differential inclusions, {\it Rocky Mountain J. Math.} {\bf 22} (1992), 117-124.

\bibitem{10} A. Cellina, V. Staicu,  On evolution equations having monotonicities of opposite sign, {\it J. Differential Equations} {\bf 90} (1991), 71-80.

\bibitem{11} D.H. Chen, R.N. Wang, Y. Zhan,  Nonlinear evolution inclusions: topological characterizations of solution sets and applications, {\it J. Functional Anal.} {\bf 265} (2013), 2039-2075.

\bibitem{12} A. Defranceschi,  Asymptotic analysis of boundary value problems for quasilinear monotone operators, {\it Asymptotic Anal.} {\bf 3} (1990), 221-247.

\bibitem{13} Z. Denkowski, S. Migorski,  Control problems for parabolic and hyperbolic equations via the theory of $G$ and $\Gamma$ convergence, {\it Annali Mat. Pura Appl.} {\bf 149} (1987), 23-39.

\bibitem{14} Z. Denkowski, S. Migorski, N.S. Papageorgiou,  On the convergence of solutions of multivalued parabolic equations and applications, {\it Nonlinear Anal.} {\bf 54} (2003), 667-682.

\bibitem{15} Z. Denkowski, S. Migorski, N.S. Papageorgiou, {\it An Introduction to Nonlinear Analysis: Theory}, Kluwer Academic Publishers, Boston (2003).

\bibitem{16} Z. Denkowski, S. Migorski, N.S. Papageorgiou, {\it An Introduction to Nonlinear Analysis: Applications}, Kluwer Academic Publishers, Boston (2003).

\bibitem{17} A. Dontchev, T. Zolezzi, {\it Well-Posed Optimization Problems}, Lecture Notes in Math., Vol. 1404, Springer-Verlag, Berlin (1994).

\bibitem{18} H. Fattorini, {\it Infinite Dimensional Optimization and Control Theory}, Cambridge Univ. Press, Cambridge UK (1999).

\bibitem{19} H. Frankowska, A priori estimates for operational differential inclusions, {\it J. Differential Equations} {\bf 84} (1990), 100-128.

\bibitem{20} L. Gasinski, N.S. Papageorgiou, {\it Nonlinear Analysis}, Chapman \& Hall/CRC, Boca Raton, FL (2006).

\bibitem{21} L. Gasinski, N.S. Papageorgiou, {\it Exercises in Analysis. Part 2: Nonlinear Analysis}, Springer, New York (2016).

\bibitem{23} S. Hu, N.S. Papageorgiou, {\it Handbook of Multivalued Analysis. Volume I: Theory}, Kluwer Academic Publishers, Dordrecht, The Netherlands (1997).

\bibitem{22} S. Hu, N.S. Papageorgiou, {\it Time-Dependent Subdifferential Evolution Inclusions and Optimal Control}, Memoirs Amer. Math. Soc., Vol. 133, No. 632 (May 1998), Providence, RI.

\bibitem{24} K. Ito, K. Kunisch,  Sensitivity analysis of solutions to optimization problems in Hilbert spaces with applications to optimal control and estimation, {\it J. Differential Equations}  {\bf 99} (1992), 1-40.

\bibitem{25} K. Ito, K. Kunisch, {\it Lagrange Multiplier Approach to Variational Problems and Applications}, SIAM, Philadelphia (2008).

\bibitem{26} A.G. Kolpakov,  $G$-convergence of a class of evolution operators, {\it Siberian Math. Journal} {\bf 29} (1988), 233-244.

\bibitem{27} J.-L. Lions, {\it Quelques M\'ethodes de R\'esolution des Probl\'emes aux Limites Non Lin\'eaires}, Dunod, Paris (1969).

\bibitem{28} Z. Liu, Existence results for evolution noncoercive hemivariational inequalities, {\it J. Optim. Theory Appl.} {\bf 120} (2004), 417-427.

\bibitem{29} P. Marcellini, C. Sbordone, {\it Dualita e perturbazione di funzionali integrali}, {\it  Ricerche Mat.} {\bf 26} (1977), 383-421.

\bibitem{31} N.S. Papageorgiou,  Sensitivity analysis of evolution inclusions and its applications to the variational stability of optimal control problems, {\it Houston J. Math.} {\bf 16} (1990), 509-522.

\bibitem{30} N.S. Papageorgiou,  A convergence result for a sequence of distributed parameter optimal control problems, {\it J. Optim. Theory Appl.} {\bf 68} (1991), 305-320.

\bibitem{33} N.S. Papageorgiou,  On parametric evolution inclusions of the subdifferential type with applications to optimal control problems, {\it Trans. Amer. Math. Soc.} {\bf 347} (1995), 203-231.

\bibitem{32} N.S. Papageorgiou,  On the variational stability of a class of nonlinear parabolic optimal control problems, {\it Zeitsch. Anal. Anwend.} {\bf 15} (1996), 245-262.

\bibitem{34} N.S. Papageorgiou, S. Kyritsi, {\it Handbook of Applied Analysis}, Springer, New York (2009).

\bibitem{35} N.S. Papageorgiou, F. Papalini, F. Renzacci, Existence of solutions and periodic solutions for nonlinear evolution inclusions, {\it Rend. Circolo Mat. Palermo} {\bf 48} (1999), 341-364.

\bibitem{repsem} D.D. Repov\v{s}, P.V. Semenov, {\it Continuous Selections of Multivalued Mappings},
Kluwer Academic Publishers, Dordrecht, The Netherlands (1998).

\bibitem{36} T. Roubicek, {\it Relaxation in Optimization Theory and Variational Calculus}, W. De Gruyter, Berlin (1997).

\bibitem{37} J. Sokolowski,  Optimal control in coefficients for weak variational problems in Hilbert space, {\it Appl. Math. Optim.} {\bf 7} (1981), 283-293.

\bibitem{38} J. Sokolowski, J.-P. Zolesio, {\it Introduction to Shape Optimization: Shape Sensitivity Analysis}, Springer-Verlag, Berlin (1992).

\bibitem{39} N. Svanstedt,  $G$-convergence of parabolic operators, {\it Nonlinear Anal.} {\bf 36} (1999), 807-842.

\bibitem{40} E. Zeidler, {\it Nonlinear Functional Analysis and its Applications II/B: Nonlinear Monotone Operators}, Springer-Verlag, New York (1990).



\end{thebibliography}
\end{document}